 \documentclass[a4paper,12pt,oneside,reqno]{amsart}

\usepackage[utf8]{inputenc}
\usepackage[T1]{fontenc}
\usepackage{hyperref}
\usepackage[headinclude,DIV13]{typearea}
\usepackage{amssymb,amsmath,amsthm}
\usepackage[usenames,dvipsnames]{xcolor}
\definecolor{violet}{rgb}{0.6,0.4,0.8}
\usepackage{enumerate,multicol,longtable,array}
\usepackage{graphicx} 
\usepackage{caption}
\usepackage{mathrsfs}
\usepackage{xcolor}
\usepackage{enumerate}
\theoremstyle{plain}
\usepackage{txfonts}
\newtheorem{theorem}{Theorem}[section]
\newtheorem{lemma}[theorem]{Lemma}
\newtheorem{proposition}[theorem]{Proposition}

\theoremstyle{definition}

\newtheorem{assumption}{Assumption}

{\itshape}{\rmfamily}

\theoremstyle{remark}

\numberwithin{equation}{section}

\newcommand{\nsided}{{\zeta}}
\renewcommand{\P}{\mathbb{P}}
\newcommand{\R}{\mathbb{R}}

\newcommand{\E}{\mathbb{E}}

\newcommand{\N}{\mathbb{N}}
\newcommand{\Z}{\mathbb{Z}}

\newcommand{\eps}{\varepsilon}

\newcommand{\ind}[1]{{\mathbf{1}}_{\left[ {#1} \right] }}

\newcommand{\ep}{\varepsilon}
\newcommand{\xbar}{\bar{x}}

\newcommand{\xmu}{x^\mu}

\newcommand{\xtep}{\tilde x^\mu}

\newcommand{\Tinit}{T^{-}}

\newcommand{\Tswap}{ T^{s}}

\newcommand\eee{\mathrm{e}}

\newcommand\Ninterval[2]{\llbracket #1,#2\rrbracket}

\numberwithin{equation}{section}

\graphicspath{ {fig/} }
\def\be{\begin{equation*}}
\def\ee{\end{equation*}}
\def\best{\begin{equation*}}
\def\eest{\end{equation*}}
\begin{document}

\title[Crossing a fitness valley]{Crossing a fitness valley as a metastable transition\\ in a stochastic population model}
%{Crossing a fitness valley for a population of stochastically varying size with competitive interactions}

\author{Anton Bovier}
\address{Institut f\"ur Angewandte Mathematik, Rheinische Friedrich-Wilhelms-Universit\"at Bonn, Endenicher Allee 60, 
53115 Bonn, Germany.}
\email{bovier@uni-bonn.de}

\author{Loren Coquille}
\address{
	Univ. Grenoble Alpes, CNRS, Institut Fourier, F-38000 Grenoble, France}
\email{loren.coquille@univ-grenoble-alpes.fr}

\author{Charline Smadi}
\address{Irstea, UR LISC, Laboratoire d'Ing\'enierie pour les Syst\`emes Complexes, 9 avenue Blaise Pascal-CS 20085, 63178 Aubi\`ere, France 
and Complex Systems Institute of Paris Ile-de-France, 113 rue Nationale, Paris, France}
\email{charline.smadi@irstea.fr}

 \keywords{ Eco-evolution; Birth and death process with immigration; Selective sweep; Coupling; competitive Lotka-Volterra system with mutations}
 \subjclass[2010]{92D25, 60J80, 60J27, 92D15, 60F15, 37N25.}

\thanks{A.B. acknowledges financial support from the German Research Foundation (DFG) 
	through the  Clusters of Excellence  \emph{Hausdorff Center for  Mathematics} and 
	\emph{ImmunoSensation},
	the Priority Programme SPP 1590 \emph{Probabilistic Structures in Evolution},
	and the Collaborative Research Center CRC 1060 \emph{The Mathematics of Emergent Effects}. L.C. has been 
	partially supported by the LabEx PERSYVAL-Lab (ANR-11-LABX-0025-01) through the 
	Exploratory Project \textit{CanDyPop} and
	by the Swiss National Science Foundation through  the grant  No.  P300P2\_161031.
	This work was also partially funded by the Chair \emph{Mod\'elisation
		Math\'ematique et Biodiversit\'e} of VEOLIA-Ecole Polytechnique-MNHN-F.X}

\maketitle

\begin{abstract}
	We consider a stochastic model of population dynamics where each individual is characterised by a trait in $\{0,1,...,L\}$ and has a natural reproduction rate, a logistic death rate due to age or competition, and a probability of mutation towards neighbouring traits at each reproduction event. 
We choose parameters such that the induced fitness landscape exhibits a valley: mutant individuals with negative fitness have to be created in order for the population to reach a trait with positive fitness. We focus on the limit of large population and rare mutations at several speeds. In particular, when the mutation rate is low enough, metastability occurs: the exit time of the valley is an  exponentially distributed random variable.
\end{abstract}

%\tableofcontents
\section{Introduction}

The biological theory of \emph{adaptive dynamics} aims at studying the interplay between ecology and evolution through the modeling  of three basic mechanisms: 
heredity, mutations, and {competition}. It was first developed in the 1990ies, partly heuristically, by {Metz, Geritz, Bolker, Pacala, Dieckmann, Law, and coauthors
	\cite {metz1995adaptive,DL96,geritz1997dynamics,BolPac1,BolPac2,DieLaw}.}

A rigorous derivation of the theory was achieved over the last decade in the context of stochastic individual-based models,
where the evolution of a population
of individuals characterised by their phenotypes under the influence of the
evolutionary mechanisms of birth, death, mutation, and ecological 
competition in an inhomogeneous "fitness landscape" is described as a measure valued Markov process.
Using various scaling limits involving large population size, small mutation rates, and 
small mutation steps, key features described in the biological theory of {adaptive dynamics}, in particular the \emph{canonical equation of adaptive dynamics (CEAD)}, the 
\emph{trait substitution sequence (TSS)}, and the \emph{polymorphic evolution sequence (PES)} were 
recovered,  
see \cite{C_CEAD,champagnat2006microscopic,FM04,C_ME,CM11,B14}. {Extensions} of those results
for more structured populations were investigated, for example, in \cite{tran_2008,leman2016convergence}.

Contrarily to the population genetics approach, individual-based models of adaptive dynamics take into account varying population sizes as well as stochasticity, which is necessary if we
aim at better understanding of phenomena involving small populations, such as mutational
meltdown \cite{coron2013quantifying}, invasion of a mutant population \cite{champagnat2006microscopic},
evolutionary suicide and rescue \cite{abu2017double}, population extinction time \cite{chazottes2016sharp,Coron2017},
or recovery phenomena \cite{baar2016stochastic,bovier2017recovery}.

The emerging picture allows to give the 
following description of the evolutionary fate of a population starting in a monomorphic initial state:
first, on a fast ecological time scale, the population reaches its ecological equilibrium. 
Second, if mutations to types of positive \emph{invasion fitness} (the invasion fitness is the average growth rate of an individual born with this trait in the 
presence of the current equilibrium population) are possible, these 
eventually happen and the population is substituted by a fitter type once a mutant trait fixates (if coexistence is not possible). 
This continues, and the monomorphic population moves according to the  TSS 
(resp. the CEAD, if mutations steps are scaled to zero) until an 
\emph{evolutionary singularity} is reached: here two types of singularities are possible: either, the 
singularity is stable, in the sense that no {further} type with positive invasion fitness can be reached,
or there are several directions with equal positive fitness that can be taken. In the latter case
the population splits into two or more sub-populations of different types which then 
continue to move on until again an evolutionary singularity is reached. If the mutation {probability} is small enough, all this happens on a time scale
of order $1/(\mu K)$, where $\mu$ is the mutation {probability} and $K$ is the carrying capacity, 
which is a measure of the maximal population size that the environment can sustain for a long time.
This process goes on until all sub-populations are located in stable evolutionary singularities. At this 
stage, no single mutation can lead to a trait with positive invasion fitness. Nonetheless, there may be
traits with positive invasion fitness that can be reached through \emph{several} consecutive
mutation steps \cite{lenski2003evolutionary,cowperthwaite2006bad}.
{Our} purpose is to present a precise analysis of how such an escape from 
a stable singularity  happens in various scaling regimes.

As we will show, three essentially different dynamics may occur. In the first one,  the mutation probability is so  
large that  many mutants (a number of order $\mu K$) are created in 
a time of order $1$. In this case the fixation time scale is dominated by the time needed for a successful mutant to invade (which is of order {$\log 1/\mu$}).
The second scenario occurs if the mutation {probability} is smaller,  but large enough so that a fit mutant will appear before the resident population dies out. In this case the fixation time scale is
exponentially distributed and
dominated by the time needed 
for the first successful mutant to be born. The last possible scenario is the extinction of the population before the fixation of the fit mutant, 
which occurs when the mutation {probability} is very small (smaller than $\eee^{-CK}$ for a constant $C$ to be made precise later).\\

{In the sequel, we denote by $\N$ the set of integers $\{1,2,3,...\}$, by $\N_0$ the set $\N \cup \{0\}$, 
	{and by $\R_+=\{x\in\R:x\geq0\}$} {the set of non negative real numbers}.
	For $n,m\in\N_0$ such that $n\leq m$, we also introduce the notation $\llbracket n,m\rrbracket:=\{n,n+1,\ldots,m\}$. }

\section{Model}

In this paper we analyse the escape problem in a specific simple model situation that, however,
captures the key {mechanisms.}
We consider a finite trait  space ${\llbracket 0,L\rrbracket}$ on which the population evolves.
To each trait $i\in {\llbracket 0,L\rrbracket}$, we assign 

\begin{itemize}
	\item a \emph{clonal birth rate}: $(1-\mu)b_i\geq 0$, where $0\leq \mu\leq 1$ is the mutation {probability};
	\item  a \emph{natural death rate}: $d_i\geq 0$.
\end{itemize}
An individual can also die from type-dependent competition.
We assign to each pair $(i,j)\in {\llbracket 0,L\rrbracket}^2$
\begin{itemize}
	\item  a \emph{competition kernel}: $c_{ij}\geq 0$,  
	where $c_{ii},c_{i0},c_{iL}>0$, for all $i\in {\llbracket 0,L\rrbracket}$.
\end{itemize}
To be able to scale the effective size of a population, the competition kernel is scaled 
down by the so-called \emph{carrying capacity}, $K$, that is,  the competitive pressure exerted by an individual of type $j$ on an individual 
of type $i$ is $c_{ij}/K$. 
Finally, to represent mutations, we assign to each pair $(i,j)\in {\llbracket 0,L\rrbracket}^2$
\begin{itemize}
	\item  a \emph{mutation kernel}: ${(m_{ij})}_{(i,j)\in{\llbracket 0,L\rrbracket}^2}$ satisfying $m_{ij}\in[0,1]$, for all 
	$(i,j)\in{\llbracket 0,L\rrbracket}^2$ and $\sum_{j\in{\llbracket 0,L\rrbracket}}m_{ij}=1$. We will focus on two cases:
	\begin{equation}\label{mut-kernel}
		{m^{(1)}_{ij}=\delta_{i+1,j} \quad\text{ or }\quad m^{(2)}_{ij}=\frac{1}{2}(\delta_{i+1,j}+\delta_{i-1,j}),}
	\end{equation}
	where $\delta_{i,j}$ is the Kronecker delta ($1$ if $i=j$, $0$ otherwise).
\end{itemize}
{We denote the stochastic process with the above mechanisms by $X$.}
The state of a population is  an element of ${\N_0}^{L+1}$.
As we will see, {before the population extinction, which is of an exponential order (see Section \ref{res_ext}),} 
the total population size has the same order as the carrying capacity {$K$.} 
Hence it will be more convenient to study the 
rescaled process $X^K=(X_0^K(t),\ldots,X_L^K(t))=X/K$ and to think of this as an element of
$\R^{L+1}$. 
Let $\mathrm{e}_i$ denote the $i$-th 
unit vector in $\R^{L+1}$. 
The generator of $X^K$ acts on bounded measurable
functions $f:\R_+^{L+1}\to \R$, for all $X^K\in({\N_0}/K)^{L+1}$, as 
\begin{align}\label{genscaled}\nonumber
	(L^{(K)}f)(X^K)=&(1-\mu)K\sum_{i=0}^L 
	(f(X^K+\mathrm{e}_i/K)-f(X^K)) b_iX^K_i \\
	&+K \sum_{i=0}^L  (f(X^K-\mathrm{e}_i/K)-f(X^K))(d_i+\sum_{j=0}^L {c_{ij}}X^K_j)X^K_i \nonumber	\\
	&+\mu K \sum_{i=0}^L \sum_{j=0}^L(f(X^K+\mathrm{e}_j/K)-f(X^K)) b_i m_{ij} X^K_i.\end{align}

A  key result, due to Ethier and Kurtz \cite{EthKur1986}, is the law of large {numbers} when $K\uparrow
\infty$ (for fixed $\mu$ and fixed time intervals), which we recall now.

\begin{proposition}[\cite{EthKur1986}, Chapter 11, Thm 2.1]
	Suppose that the initial conditions converge in probability to a deterministic
	limit, i.e. $\lim_{K\to\infty} X^K(0)=x(0)$. Then, for each $T\in \R_+$,
	the rescaled process $(X^K(t), 0\leq t\leq T)$
	converges in probability, as
	$K\to\infty$, 
	to the deterministic process $x^\mu=(x_0^\mu,\ldots,x_L^\mu)$ which is the
	unique solution to the following  dynamical system: 
	\begin{align}\label{cde-gen}
		\frac{d
			x_i^\mu}{dt}=\left((1-\mu)b_i-d_i-\sum_{i=0}^Lc_{ij}x_j^\mu\right)x_i^\mu
		+\mu \sum_j m_{ji} b_j x^\mu_j,
		\quad i=0,\ldots,L,
	\end{align}
	with initial condition $x(0)$. 
\end{proposition}	

There will be two important quantities associated with our processes.
The equilibrium density of a monomorphic $i$-population is
\begin{equation} \label{defbarn}
	\bar{x}_{i}: =\frac{b_i -d_i}{c_{ii}} \vee 0.
\end{equation}
The effective growth rate (or selective advantage or disadvantage) of a small mutant population with trait $i$ 
in a $j$-population at equilibrium, 
is the so-called \emph{ invasion fitness},
$f_{ij}$, given by
\begin{equation} \label{deffitinv}
	f_{ij} := b_i -d_i - c_{ij}\bar{x}_j.
\end{equation}

The importance of the above two quantities follows from the properties of the limiting competitive Lotka-Volterra 
system \eqref{cde-gen} with $\mu=0$.
Namely, if we assume
\begin{equation}\label{condfitnesses}
	\bar{x}_{1}= \frac{b_1-d_1}{c_{11}} >0,\quad \text{and} \quad f_{01}<0<f_{10},
\end{equation}
then the system \eqref{cde-gen} with $\mu=0$ and $L=1$ has a unique stable equilibrium, $(x_0=0,x_1=\bar{x}_1)$, and two unstable steady states, 
$(x_0=\bar{x}_0,x_1=0)$ and $(x_0=0,x_1=0)$.

We are interested in the situation where $\bar x_0>0$, 
$f_{i0}<0$, $1 \leq i \leq L-1$, $f_{L0}>0$, and $f_{0L}<0$. 
Under these assumptions, all mutants created by the initial population 
initially have a negative growth rate,  and thus tend to die out. However, 
if by chance such mutants survive long enough to give rise to further mutants, such that eventually
an individual will reach the trait $L$, it will found a population at this trait that, with positive probability,
will grow and eliminate the resident population through competition. Our purpose is to analyse 
precisely how this process  happens. The process that we want to describe can 
be seen as a manifestation of the phenomenon of \emph{metastability} (see, e.g., the recent monograph 
\cite {BH15} and references therein). The initial population appears stable for a long time and makes repeated attempts to send mutants to the trait $L$, 
which will eventually be reached and take over the 
entire population. As we will see, this  leads to several features known from metastable 
phenomena in other contexts: exponential laws of the transition times, 
fast realisation of the final ``success run", and the realisation of this run by a ``most likely" realisation. 
As {usual} in the context of metastability, we  need a scaling parameter to make precise 
asymptotic statements. In our case this is the \emph{carrying capacity}, $K$, 
which allows to scale the population size to infinity.
Apart from scaling the population size by taking $K\uparrow\infty$, we are also  interested in the 
limit of small  mutation {probabilities,} $\mu=\mu_K\downarrow 0$, with possibly simultaneous time rescaling. 
This  gives rise to essentially different asymptotics, depending on how $\mu$ tends to zero as a function 
of $K$.

\section{Results} 

Before stating our main results, let us  make our assumptions precise:
\begin{assumption}\label{ass.1}
	\begin{enumerate}
		\item[$\bullet$] Viability of the resident population:
		$ \bar{x}_0>0. $
		\item[$\bullet$] {Fitness valley:} All traits are unfit with respect to 0 except $L$:
		\begin{equation} \label{A1} 
			f_{i0}<0, \text{ for } i\in\llbracket 1,L-1\rrbracket  \text{ and } f_{L0}>0.
		\end{equation}
		\item[$\bullet$] All traits are unfit with respect to $L$:
		\begin{equation} \label{A2}
			f_{iL}<0, \text{ for  } i\in\llbracket 0,L-1\rrbracket.
		\end{equation}
		\item[$\bullet$]  The following fitnesses are different:
		\begin{align}
			f_{i0}&\neq f_{j0}, \text{ for all }i\neq j,\label{inequ-fit-1}\\
			f_{iL}&\neq f_{jL} ,\text{ for all }i\neq j.\label{inequ-fit-2}
		\end{align}
	\end{enumerate}
\end{assumption}

\begin{figure}[h!]
	\centering
	\includegraphics[width=.4\textwidth]{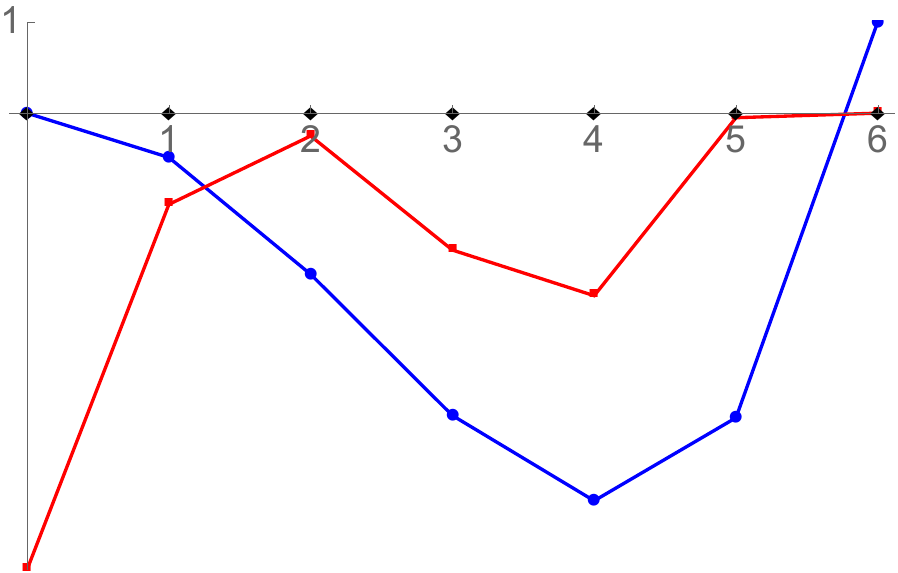}
	\caption {Example of {a} fitness landscape satisfying Assumption \ref{ass.1} with $L=6$. Blue curve: $i\mapsto f_{i0}$, red curve: $i\mapsto f_{iL}$.}
	\label{fitness}
\end{figure}

{Note that conditions \eqref{inequ-fit-1} and \eqref{inequ-fit-2} are imposed in order to lighten the analysis of the deterministic system (see Lemma \ref{lemma}). Similar results are probably true without these assumptions 
	but the proofs would be {unnecessarily} more technical. Similar {hypotheses} are made in the article \cite{DurMay2011}. 
}

Before proceeding to the statements of our results, let us show that Assumption \ref{ass.1} can be 
realised with well-chosen birth-, death-, and competition rates. 
A possibility is to fix birth and death rates associated to every trait to be $1$ and $0$, respectively. 
In that case, Assumption \ref{ass.1}  imposes  constraints  on the competition rates 
$(c_{i0})_{i\in\llbracket 1, L\rrbracket}$ 
and 
$(c_{iL})_{i\in\llbracket 0, L-1\rrbracket }$, which must be equal to $(1-f_{i0})_{i\in\llbracket 1, L\rrbracket}$ and 
$(1-f_{iL})_{i\in\llbracket 0, L-1\rrbracket}$, respectively. 
We  complete the competition matrix by taking symmetric values 
(except for $c_{0L}$ and $c_{L0}$ which are now fixed and different) and by choosing $c_{ij}=1$, for all 
pairs $(i,j)\in{\llbracket 1, L-1\rrbracket}^2$.

\subsection{Deterministic limit $ (K,\mu) \to (\infty,\mu)$, then $\mu\to 0$} \label{sec-det}

The first regime we are interested in is the case when $\mu$ is small but does not scale with the population size. From a biological point of view, 
this corresponds to high mutation {probabilities.} %as the dynamics of mutations can be observed at the population scale.
Note that a similar scaling {has been} studied in \cite{BovWan2013} {and \cite{DurMay2011}.} 
In both papers, the context was very different since these authors considered
the arrival of \emph{fitter} rather than unfitter mutants, as we do here.
In \cite{BovWan2013}, individuals only {suffer} competition from the nearest neighbouring traits.
{In \cite{DurMay2011}, an  exponentially growing population of tumor cells is modeled by a Moran model {with immigration},
	and back mutations are not considered.} 

%{The analysis of the limiting deterministic system has some 
%similarities but the concentration of the (rescaled) stochastic system is quite different due to the very different fitness 
%hypotheses.}

\begin{theorem}\label{thm-piecewise-constant}
	Suppose that Assumption \ref{ass.1} 
	holds. 	
	Take as initial condition 
	\begin{equation}
		x^\mu(0)=(\xbar_0,0,\ldots,0). 
	\end{equation}
	
	Then, for $i \in {\llbracket 0,L\rrbracket}$, as $\mu\to0$,
	uniformly on bounded time intervals, 
	\begin{equation}\label{lim-proc}
		\frac{\log\left[x_i^{\mu}\left(t\cdot\log\left(1/\mu\right)\right)\right]}{\log(1/\mu)}
		\to
		x_i(t),
	\end{equation}
	where $x_i(t)$ is piece-wise linear. More precisely, 
	\begin{enumerate}
		\item in the case of  1-sided mutations, $m_{ij}=m^{(1)}_{ij}$, for ${i\in\llbracket 0, L-1\rrbracket}$,
		\begin{align} 
			x_i(t)&=\left\{
			\begin{array}{ll}
				-i, 															
				& \text{for } 0\leq t <  L/f_{L0},\\
				-i- (t-L/f_{L0})	\min_{k\in\Ninterval{0}{i}}	|f_{kL}|,	&\text{for } t>L/f_{L0},
			\end{array}
			\right.\\\nonumber
			and \qquad \qquad \qquad \qquad &\\
			x_L(t)&=\left\{
			\begin{array}{ll}
				-L+f_{L0}t,			& \text{for } 0\leq t < L/f_{L0},\\
				0,		&\text{for }  t>L/f_{L0}.
			\end{array}
			\right.
		\end{align}
		\item in the case of  2-sided mutations, $m_{ij}=m^{(2)}_{ij}$:
		consider the sequence $\{i_1,\ldots,i_r\}$ of ``fitness records", defined recursively by $i_1=0$, 
		$i_k=\min\{{i\in\llbracket 0, L-1\rrbracket}: f_{iL} <f_{i_{k-1}L} \}$,
		\begin{equation}
			x_i(t)=\begin{cases}	
				-i \vee (-L-(L-i)+f_{L0}t) , 			
				& \text{for } 0\leq t < L/f_{L0},\\
				-(L-i)
				\vee
				\max_{k\in\llbracket0, i \rrbracket }	\{-i-|f_{kL}|(t-L/f_{L0})\}		\\
				\hspace{1.5cm}\vee
				\max_{k\in\llbracket 1, r \rrbracket }	\{-i_k-|i-i_k|-|f_{i_kL}|(t-L/f_{L0})\},			
				&\text{for } t>L/f_{L0}.
			\end{cases}
		\end{equation}
	\end{enumerate}
	{Moreover, 
		\begin{align}
			{\left(x_0^\mu(t\log (1/\mu)),x_L^\mu(t\log (1/\mu))\right)\to
				\left\{
				\begin{array}{ll}
					(\xbar_0,0),		& \text{for } 0\leq t < L/f_{L0},\\
					(0,\xbar_L),		&\text{for }  t>L/f_{L0}.
				\end{array}
				\right.}
		\end{align}
	}
\end{theorem}
The shape of $x(t):=(x_0(t),\ldots,x_L(t))$ can be seen on Figures \ref{fig-pow-1sided} and \ref{fig-pow-general} 
in the 1-sided and 2-sided cases, respectively.

In the 1-sided case, the rescaled deterministic process $x(t)$ can be explained as follows: {In the first phase, the 0-population stays close to $\xbar_0$ until the $L$-population reaches order one. As competition between the populations
	of type $i$ and $j$ for $i, j\neq 0$ is negligible in comparison to competition between type $i$ and type 0,} {for $i \in \llbracket 1, L \rrbracket$,} the $i$-population 
first stabilises around $O(\mu^i)$ in a time of order {$o(1)$,} then the $L$-population, {starting from a size $O(\mu^L)$,} 
grows exponentially with rate $f_{L0}$ until reaching order one 
(which {takes} a time ${L}/{f_{L0}}$) while the other types stay stable. Next, a swap between populations 
$0$ and $L$ 
(two-dimensional Lotka-Volterra system) is happening in a time of order {$o(1)$,} and finally, for $i\neq L$, 
the $i$-population decays exponentially from 
$O(\mu^i)$ with a rate given by the lowest (negative) fitness of its left neighbours, $(\min_{j\in\Ninterval{0}{i}}|
f_{jL}|)$ {while the $L$-population approaches its equilibrium density $\xbar_L$.}
{To understand the rate of decrease during the last phase, let us consider only the $0$- and $1$-
	populations. The competition exerted by populations 
	$j \in \llbracket 0,L-1 \rrbracket $ on the $0$- and $1$-populations is negligible with respect to the competition 
	exerted by the $L$-population, which has 
	a size of order $1$. As a consequence, $x_0^\mu$ has a dynamics close to this of the solution to:
	\be
	\dot{\tilde{x}}_0(t)=f_{0L}\tilde{x}_0(t),
	\ee
	that is to say, {$x_0(t) \approx x_0(0)\eee^{f_{0L}t}$}, and $x_1$ has a dynamics close to this of the solution to:
	\be
	\dot{\tilde{x}}_1(t)=f_{1L}\tilde{x}_1(t)+ \mu\tilde{x}_0(t)=f_{1L}\tilde{x}_1(t)+ \mu\tilde{x}_0(0)\eee^{f_{0L}t} ,
	\ee
	that is to say
	\be
	x_1(t) \approx x_1(0)\eee^{f_{1L}t} + \mu \frac{x_0(0)}{f_{0L}-f_{1L}}\left(\eee^{f_{0L}t}-\eee^{f_{1L}t}\right). 
	\ee
	From this heuristics we get that
	\begin{align} \nonumber x_1^\mu(t \log (1/\mu)) &\approx x_1^\mu(0)\mu^{|f_{1L}|t} + \mu \frac{x_0^\mu(0)}{|f_{1L}|-|f_{0L}|}
		\left(\mu^{|f_{0L}|t}-\mu^{|f_{1L}|t}\right)\\
		& = \mu \left(C_1\mu^{|f_{1L}|t} + \frac{C_0}{|f_{1L}|-|f_{0L}|}\left(\mu^{|f_{0L}|t}-\mu^{|f_{1L}|t}\right)\right), \end{align}
	where $C_0$ and $C_1$ are of order $1$.
	We thus see that the leading order is $\mu^{1+\inf\{ |f_{0L}|, |f_{1L}| \}t}$. Reasoning in the similar way for the other populations yields that 
	the leading order for the variation of the $i$-population size ($ i \in \llbracket 0,L-1 \rrbracket$) is 
	$\mu^{i+\inf\{ |f_{0L}|, |f_{1L}|,...,|f_{iL}| \}t}$.}

In the 2-sided case, a modification of the order of magnitude of the $i$-population (for {$i\neq L$)} happens due to backward mutations. 
{The reasoning is similar to the heuristics we have just described, except that mutants from the $i$-population ($i \in \llbracket 1,L\rrbracket$) 
	might also have an impact on the decrease rate of the $(i-1)$-population. This is the case if $x_i^\mu/x_{i-1}^\mu \geq C/\mu$, for a positive constant $C$. 
	Under this condition 
	the number of type-$(i-1)$ individuals produced by mutations of type $i$-individuals has the same order as the type $(i-1)$ population size.}

\subsection{Stochastic limit $(K,\mu)\to(\infty,0)$} \label{res_sto}

When the mutation {probability} is small, the dynamics and time scale of the invasion process  depends on the scaling of the mutation probability per reproductive 
event, $\mu$, with respect to the carrying capacity $K$. We consider in this section {mutation probabilities with two possible forms. Either,}
\begin{equation}\label{muenpuissancedeK}
	{ \mu= f(K) K^{-1/\alpha}, \quad \text{with} \quad  \alpha \geq 1 \quad \text{and} \quad |\ln f(K)| = o(\ln K),}
\end{equation}
{or}
\begin{equation}
	{ \mu = o(1/K).}
\end{equation}
For simplicity, in Sections 
\ref{res_sto} and \ref{res_ext} we only consider the mutation kernel $m_{ij}^{(1)}= \mu \delta_{i+1,j}$. 
% Taking into account back mutations in these regimes would bring more technicalities and would not 
% modify the macroscopic dynamics (see Section 4.3 in \cite{Smadi2017} for a related discussion).

For $v\geq 0$ and $0\leq i \leq L$,   let  $T_v^{(K,i)}$ denote  the first  time the $i$-population  reaches the  size 
$\lfloor v K\rfloor$, 
\begin{equation}\label{defTepsKM}
	T^{(K,i)}_v := \inf \{ t \geq 0, X_i(t)= \lfloor v K \rfloor \}.
\end{equation}

In a time of order one, there will be of order {$K\mu^i$} mutants of 
type $i$, provided that this number is larger than $1$. In particular, there will be
of order {$K\mu^L$} fit $L$-mutants at time one, if $L/\alpha <1$. This is the regime of large mutation {probability.}
In this case, the time for the $L$-population to hit a size of order $K$ 
is of order $\log K$. We obtain a precise estimate of this time, as well as of the time for the trait $L$ to 
outcompete the other traits under 
the same assumptions.
Let us introduce 
\be\label{timetoextinction}
t(L,\alpha):= \frac{L}{\alpha} \frac{1}{f_{L0}}+ \sup \left\{\left(1 - \frac{i}{\alpha}\right) \frac{1}{|f_{iL}|}, 0 \leq i \leq L-1\right\}, 
\ee
and the time needed for the   populations at all sites  but $L$ to get extinct,
\begin{equation}\label{defT0ttsaufM}
	T^{(K,\Sigma)}_0 := \inf \Big\{ t \geq 0, \sum_{0\leq i \leq L-1}X_i(t)= 0 \Big\}.
\end{equation}
With this notation we have the following asymptotic result.

\begin{theorem}
	\label{pro_phase1_mugrand}
	Assume that \eqref{muenpuissancedeK} holds and that $L < \alpha < \infty$.
	Then there exist two positive constants $\eps_0$ and $c$ such that, for every $0<\eps \leq \eps_0$,
	\begin{equation}\label{eq1th}
		\liminf_{K \to \infty} \P \left( (1-c\eps)\frac{1}{\alpha} \frac{L}{f_{L0}}< 
		\frac{T^{(K,L)}_\eps }{\log K} { <\frac{T^{(K,L)}_{\bar{x}_L-\eps} }{\log K}}
		< (1+c\eps)\frac{1}{\alpha} \frac{L}{f_{L0}} \right) \geq 1-c\eps.
	\end{equation}
	{Moreover,
		\begin{equation}\label{eq2th}
			\frac{T^{(K,\Sigma)}_{0}}{\log K} \to t(L,\alpha), \quad \text{in probability}, \quad (K \to \infty) 
		\end{equation}
		and} there exists a positive constant $V$ such that
	\begin{equation} \label{eq3th}   
		\limsup_{K \to \infty} \P \left( \sup_{t \leq \eee^{KV}}\left| 
		X_L \left(T^{(K,L)}_{\bar{x}_L-\eps}+t \right)-\bar{x}_L K  \right|
		> c\eps K\right) \leq c\eps. \end{equation}
\end{theorem}

In other words, it takes a time of order $t(L,\alpha) \log K$ for the $L$-population to outcompete the other populations and enter in a neighbourhood 
of its monomorphic equilibrium size $\bar{x}_L K$. Once this has happened,  it stays close to this equilibrium for at least a time $\eee^{KV}$, where 
$V$ is a positive constant.

Note that the constant $t(L,\alpha)$ can be intuitively computed from the deterministic limit. Indeed, 
for $\alpha>L$, we prove that the system performs small 
fluctuations around the deterministic evolution studied above: the $i$-population first stabilises around {$O(K\mu^i)$} in a time of order one, then 
the $L$-population grows exponentially with rate $f_{L0}$ 
until reaching order $K$ ({{super-critical branching process,} needs a time close to $L\log K /(\alpha f_{L0})$) while the other types stay stable, 
	the swap between populations $0$ and $L$ then {takes} a time of order one, and finally, for $i\neq L$, the $i$-population decays exponentially 
	from {$O(K\mu^i)$} 
	to extinction with a rate given by the lowest (negative) fitness of its 
	left neighbours ({sub-critical branching process,} needs a time close to $(\sup_{j\in\Ninterval{0}{i}}(1-j/\alpha)/|f_{jL}|)\log K$). Thus the time until 
	extinction of all non-$L$ populations is close to the constant \eqref{timetoextinction} times $\log K$.
	
	{Note that Theorem \ref{pro_phase1_mugrand} is close in spirit to the results of Durrett and Mayberry in \cite{DurMay2011}, and some of our techniques of 
		proof are similar.  However the processes they consider differ from ours at many levels. More precisely, they consider a Moran model with either 
		fixed or growing population size (with a growth independent of the composition in traits of the population), 
		and mutants with increasing fitnesses, while we work with a model with varying population size (where variations 
		depend on the population composition via trait dependent competitive interactions) and allow negative fitnesses.
		Moreover in \cite{DurMay2011}, all mutations have the same effect and back mutations are not considered, whereas it leads 
		to interesting behaviour and more technicalities in our case. Finally, the way mutations are encoded in Moran like 
		models do not allow to distinguish between effects due to birth rate, death rate, and competition. The class of models we 
		consider allow a much wider variety of mutations (see Section 3 in \cite{rebekka2015genealogies} for a detailed discussion 
		on these aspects).}
	
	Next we consider the case of small mutation {probability,}  when $L/\alpha>1$. In this case, there is no $L$-mutant
	at time one, and the fixation of the trait $L$ happens on a much longer time scale.
	{In this section, we are interested in the case where the mutation $L$ goes to fixation with a probability close to one. In particular, the first 
		$L$-mutant has to be born before the extinction of the population.}
	
	We define, for $0 < \rho <1$,
	\be
	\lambda(\rho):= \sum_{{k=1}}^\infty \frac{(2k)!}{{ (k-1)! (k+1)!}} \rho^k \left( 1-\rho \right)^{k+1},
	\ee
	and, for $\lfloor \alpha \rfloor +1 \leq i \leq L-1$, set
	$ \rho_i:= b_i/(b_i + d_i + c_{i0}\bar{x}_0)$.

	\begin{theorem} \label{pro_mupetit}
		\begin{itemize}
			\item[$\bullet$] {
				Assume that \eqref{muenpuissancedeK} holds, $\alpha \notin \N$ and ${1<} \alpha < L$.}
			Then there exist {two} positive constants $\eps_0$ and $c$, and two exponential random variables $E_-$ and $E_+$ 
			with  parameters   
			\be
			(1+c\eps){\frac{ \bar{x}_0b_0... b_{\lfloor \alpha \rfloor-1}}{|f_{10}|...|f_{\lfloor \alpha \rfloor 0}|}} \frac{f_{L0}}{b_L} \prod_{i=\lfloor \alpha \rfloor +1}^{L-1}\lambda(\rho_i)
			\qquad \text{and} \qquad 
			(1-c\eps){\frac{ \bar{x}_0b_0... b_{\lfloor \alpha \rfloor-1}}{|f_{10}|...|f_{\lfloor \alpha \rfloor 0}|}} \frac{f_{L0}}{b_L} \prod_{i=\lfloor \alpha \rfloor +1}^{L-1}\lambda(\rho_i), 
			\ee
			such that, for every $\eps \leq \eps_0$,
			\begin{equation}
				\liminf_{K \to \infty} \P \left( E_-\leq 
				\frac{T^{(K,L)}_{\bar{x}_L-\eps} \vee T^{(K,\Sigma)}_{0}}{K\mu^L} \leq 
				E_+ \right) \geq 1-c\eps.
			\end{equation}
			\item[$\bullet$] {There exists a positive constant $V$ such that if $\mu$ satisfies 
				\be
				{K\mu \ll  1 \text{ and } e^{VK}\gg 1/K\mu^L },
				\ee
				then the same conclusion holds, with the corresponding parameters, for $E_-$ and $E_+$:
				\be
				(1+c\eps)\bar{x}_0 \frac{f_{L0}}{b_L} \prod_{i=1}^{L-1}\lambda(\rho_i)
				\qquad \text{and} \qquad 
				(1-c\eps)\bar{x}_0 \frac{f_{L0}}{b_L} \prod_{i=1}^{L-1}\lambda(\rho_i).
				\ee }
		\end{itemize}
		Moreover, under both assumptions, there exists a positive constant $V$ such that
		$$  
		\limsup_{K \to \infty} \P \left( \sup_{t \leq \eee^{KV}}\left| 
		X_L \left(T^{(K,L)}_{\bar{x}_L-\eps}+t \right)-\bar{x}_L K  \right|
		> c\eps K\right) \leq c\eps. 
		$$
	\end{theorem}
	
	In the {first case,} the typical trajectories of the process are as follows: mutant populations of 
	type $i$, for $1 \leq i \leq \lfloor \alpha \rfloor$, 
	reach a size of order $K\mu^i \gg 1$ in a time of order ${b_{i-1}\log K/f_{i0}} $ {(they are well approximated by birth-death processes with immigration and their behaviour is then close to the deterministic limit),} and mutant populations of type $i$, 
	for $\lfloor \alpha \rfloor 
	+1 \leq i \leq L$, describe a.s.\ finite excursions, 
	whose a proportion of order $\mu$ produces a mutant of type $i+1$. 
	Finally, every $L$-mutant has a probability $f_{L0}/b_L$ to produce a population which 
	outcompetes all 
	other populations.
	The term $\lambda(\rho_i)$ is the expected number of individuals in an excursion of a subcritical
	birth and death process of birthrate $b_i$ and death rate $d_i+c_{i0}\bar x_0$ {excepting the first individual.}
	Hence $\mu \lambda(\rho_i)$ is the approximated probability for a type $i$-population 
	$(\lfloor \alpha \rfloor +1 \leq i \leq L-1)$ to produce a mutant of type $i+1$,
	{and} the overall time scale can be recovered as follows:
	\begin{enumerate}{
			\item The last 'large' population is the $\lfloor \alpha \rfloor$-population, which {reaches} a size of order 
			{$K\mu^{\lfloor \alpha \rfloor}$} after a time
			which does not go to infinity with $K$.
			\item The $\lfloor \alpha \rfloor $-population produces an excursion of an  $(\lfloor \alpha \rfloor+1 )$-population at a rate of order 
			{$K\mu^{\lfloor \alpha \rfloor+1}$, which} has a probability of order $\mu$ to produce an excursion of a  
			$(\lfloor \alpha \rfloor+2 )$-population, and so {on,}}
	\end{enumerate}
	giving the order {$K\mu^L$}.\\
	
	{Notice that Theorem \ref{pro_mupetit} implies that, for any mutation rate which converges to zero more slowly than 
		$\eee^{-VK}/K$, the population  crosses 
		the fitness valley with  probability tending to 1, as $K\to\infty$. Our results thus cover a wide range of biologically relevant 
		cases.}
	
	{In fact, we believe that the results hold  as long as $K\mu \gg \rho_0(K)$, where $\rho_0(K)$ is the inverse of the mean extinction time 
		of the $0$-population starting at its quasi-stationary distribution (see the next section for a precise definition). However, we are not able to control 
		precisely enough the law of $X_0$ before its extinction (but see \cite{chazottes2016sharp} for results in this direction).}\\
	
	{We also think} that $\alpha \notin \N$ is only a technical assumption which could be suppressed but would 
	bring 
	more technicalities into the proof. Namely, in this case, 
	the $\lfloor\alpha \rfloor$ population size would not be large, but of order one, and we would have to 
	control its size more carefully.
	
	\subsection{On the extinction of the population}\label{res_ext}
	
	One of the key advantages of stochastic logistic birth and death processes on constant size processes when dealing with population genetics issues is that 
	we can compare the time scale of mutation processes and the population lifetime.
	In particular, for the case of fitness valley crossing, we can show that if the mutation probability $\mu$ is too small, 
	the population gets extinct before the birth of the first mutant of type $L$.
	
	The quantification of the lifetime of populations with interacting individuals is a tricky question (see \cite{chazottes2016sharp,chazottes2017time} 
	for recent results)
	and we are not able to determine necessary and sufficient conditions for the $L$-mutants to succeed in invading before the population extinction. 
	However, we provide some bounds in the next results.
	
	{The previous theorem (Theorem \ref{pro_mupetit}) provided 
		a wide range of mutation {probabilities} $\mu$ for which the type $L$ mutant fixates.
		The following theorem (Theorem \ref{theo_range})}
	provides a small range for which the population dies before the birth of the first $L$-mutant. Before stating it, we 
	{introduce a parameter scaling the extinction time of the $0$-population,
		\begin{equation}\label{defCpara}
			\rho_0(K):=\sqrt{K}\exp\left(-K(b_0-d_0+d_0\ln (d_0/b_0))\right).
		\end{equation}
		More precisely, it is stated in \cite{chazottes2016sharp} that $\E_\nu[T_0^{(mono)}]=1/\rho_0(K)$, where $\nu$ is the stationary distribution of a monomorphic 
		$0$-population, and $T_0^{(mono)}$ its extinction time.
		We also need to introduce}
	the two stopping times
	\be
	T_0:= \inf \{t \geq 0, X_i(t)=0, \forall \ 0 \leq i \leq L\} \quad \text{and} \quad B_L := \inf \{t \geq 0, X_L(t)>0\},
	\ee
	{as well as the following assumption.}
	\begin{assumption} \label{A4} The birth- and death-rates satisfy the conditions
		\begin{equation} 
			b_i < d_i, \quad 1 \leq i \leq L-1.
		\end{equation}
	\end{assumption}
	Then we have the following result:
	
	\begin{theorem} \label{theo_range}
		Suppose that Assumption \ref{ass.1} holds. 
		\begin{enumerate}
			\item If
			$ K\mu \ll \rho_0(K)$,
			then 
			$ \P \left( T_0 < B_L  \right) \underset{K \to \infty}{\to} 1$.
			\item If Assumption \ref{A4} holds and
			$ K\mu^L \ll \rho_0(K)$,
			then 
			$ \P \left( T_0 < B_L  \right) \underset{K \to \infty}{\to} 1$.
		\end{enumerate}
	\end{theorem}

	{If $ K\mu^L \ll \rho_0(K)$ but} the intermediate mutants are fitter, the pattern is less clear. For instance, one of the 
	intermediate mutants {could} fix before being replaced (or not) by the type $L$ mutant.

	\section{Generalisations}
	
	Our results can be generalised to the following settings:
	\begin{itemize}
		\item If the fitness landscape is such that \emph{coexistence} is allowed between populations of traits $0$ and $L$, i.e.\ if {$f_{L0}>0$} and 
		$f_{0L}>0$, then the analysis of {the} invasion phase is the same, but the fixation phase differs in such a way that traits $0$ and $L$ become macroscopic and stabilise around their common equilibrium $(n^*_0,n^*_L)$, the non trivial fixed point of the 2-species Lotka-Volterra system. Moreover, the unfit mutant populations stay microscopic if we assume $ f_{i,\{0,L\}}:=b_i-d_i-c_{i0}n^*_0-c_{iL}n^*_L<0$, for all $i=1,\ldots,L-1.$ In the 1-sided case, those stay of order $K\mu^i$, while in the 2-sided case, they stay of order $K\mu^{\min\{i,L-i\}}$. There is no complicated decay phase as in Section \ref{par-after-swap-2sided}, and its stochastic analog.
		\item If the mutation {probability} $\mu$ depends on the trait $i$, while still fulfilling the prescribed scalings associated to our different theorems, those still hold.
		\item Consider the biologically relevant case (especially for cancer) where deleterious mutations accumulate until a 
		mutant individual gathers $L$ different mutations, in which case it becomes fit. Each individual bearing $k$ mutations 
		can then be labeled by the trait $k$. The main difference with our setting is that there are now $L!$ ways of reaching 
		an individual of trait $L$ with a sequence of $L$ mutations. Thus, the invasion time of the population $L$ is divided by 
		$L!$ {in the small mutation regime (Theorem \ref{pro_mupetit}) and will stay the same in 
			the large mutation regime (Theorem \ref{pro_phase1_mugrand})}.
	\end{itemize}
	
	\section{Biological context}
	
	The existence of complex phenotypes often involve interactions between different genetic loci. This can lead to cases, where 
	a set of mutations are individually deleterious but in combination
	confer a fitness benefit. To acquire the beneficial genotype, a population must
	cross a fitness valley by first acquiring the deleterious intermediate mutations. Empirical examples of such phenomena  have been found 
	in bacteria \cite{schrag1997adaptation,maisnier2002compensatory} and in viruses \cite{o1984vesicular,giachetti1988altered}, for instance.
	
	To model those phenomena,  several authors considered the case of the sequential fixation of intermediate mutants, as it 
	appeared to be the most likely scenario to get to the fixation of the 
	favorable mutant \cite{wright1965factor,wade1991wright,moore1994simulation}, 
	especially when the population size is small or the mutants neutral or weakly deleterious.

	A scenario where a combination of mutations fixates simultaneously without the prior fixation of one 
	intermediate mutant was first suggested by Gillepsie 
	\cite{gillespie1984molecular}. He observed that the rate of production of fit genotypes is proportional 
	to the population size, and because in the population genetic 
	models the probability of fixation of a beneficial allele is  independent of the population size, he 
	deduced that the expected time for the 
	fixation of the fit mutant decreases as  population size increases. Thus it could be a likely 
	process in the evolution of large populations. This scenario, 
	called stochastic {tunneling} by Iwasa and coauthors \cite{Iwasa2004}, has been widely studied since
	then (see 
	\cite{carter2002evolution,Weinreich2005,Weissman2009,Gokhale2009,haeno2013stochastic} and 
	references therein) by means of constant size population genetic models.
	But the use of such models hampers taking into account several phenomena. 
	
	First, an important 
	question is the lifetime of the population under study. 
	If the mutation {probability} is too small, the population can get extinct before the appearance of the first 
	favourable mutant. Imposing a constant (finite or infinite) 
	population size is thus very restrictive in this respect. 
	In the case of logistic processes that we are studying in this work, the total population size typically remains in the order
	of the carrying capacity $K$ during a time of order $\eee^{KV}$ (with $V$ a positive constant depending on the model's parameters), 
	before getting extinct.
	
	Second, in population genetic models, a fitness is assigned to each type, independently of 
	the population state. In the case of the Moran model, which is used in the series of papers we just 
	mentioned, the probability for a given individual to be picked 
	to replace an individual who dies is proportional to its fitness. If we want to compare our result with this setting, we  have to assume:
	\be
	b_i=b \quad \text{and} \quad |f_{ij}|=|f_{ji}|, \quad \forall 0 \leq i,j \leq L, 
	\ee
	thus restricting the type of fitnesses we could take into account (see Section 3 in \cite{rebekka2015genealogies} for a detailed discussion on this topic).

	Another series of papers
	\cite{Iwasa2004a,serra2006waiting,Serra2007,sagitov2009multitype,alexander2013conditional} focuses 
	on initially large 
	populations doomed for  rapid extinction (for instance cancer cells subject to chemotherapy, or 
	viruses invading a new host while not being adapted to it), 
	except if they manage to accumulate mutations to produce a fit variant (for instance resistant to treatments). 
	The authors use multi-type branching processes. This approach has the advantage 
	to lead to explicit expressions, as the branching property makes the calculations easier, but has two 
	main drawbacks: first it neglects interactions between 
	individuals, whereas it is well known that they are fundamental in processes such as tumor growth; 
	second, 
	branching processes either go to extinction or survive 
	forever with an  exponentially growing size, which is not realistic for biological populations.

	A last point we would like to comment is the possibility of back mutations. 
	They are ignored in all  papers we mentioned, usually accompanied with the argument that they would not have a 
	macroscopic effect on the processes under consideration.
	However, it has been shown that, when the mutation {probabilities} are large enough, scenarios where some
	loci are subject to two successive opposite mutations are likely to 
	be observed (for an example, see \cite{DePristo2007}). 
	This is why we included the possibility of back mutations in the case of high mutation {probabilities} in 
	Section \ref{sec-det}.
	
	\section{Proof of Theorem  \ref{thm-piecewise-constant}}
	
	We give the detailed  proof for $L$ even {and mention the modifications which have to be made  for $L$ odd during the proof.}
	A key step in the proof of Theorem \ref{thm-piecewise-constant} is the  following lemma.

	\begin{lemma}\label{lemma}
		Let $\nsided\in\{1,2\}$,  $(b_0,\ldots,b_{L})\in(\R^+)^{L+1}$, $(\ell_0,\ldots,\ell_{L})\in(\R^+)^{L+1}, (p_0,\ldots,p_L)\in({\R^+})^{L+1}$ and 
		$(f_0,\ldots,f_L)\in\R^{L+1}$ such that 
		\begin{align}
			f_i\neq f_j,\quad &\text{for all} \quad i\neq j. \label{diag}
		\end{align}
		Let
		\begin{equation}\label{def-matrix}
			M_\nsided(\mu,L):=\left(
			\begin{array}{cccccc}
				f_0-b_0\mu&0&0&0&0\\
				\frac\mu\nsided b_0&f_1-b_1\mu&0&0&0\\
				0&\frac\mu\nsided b_1&f_2-b_2\mu&0&0\\
				0&0&\ddots&\ddots&0\\
				0&0&0&\frac\mu\nsided b_{L-1}&f_{L}-b_L\mu
			\end{array}
			\right).
		\end{equation}
		Then the solution to the linear system 
		\begin{equation}\label{triangle-syst}
			\frac {dy}{dt}=M_\nsided(\mu,L)y,
		\end{equation}
		with initial condition 
		\begin{equation}
			y(0)=(\ell_0\mu^{p_0},\ldots,\ell_L\mu^{p_L}) ,
		\end{equation}
		satisfies
		\begin{equation} \label{defmi}
			\lim_{\mu\to0}\frac{\log(y_i(t\log(1/\mu)))}{\log(1/\mu)}=
			-m_i(t):=-\min_{\substack{\gamma,\alpha\in\llbracket 0,L\rrbracket:\\ \ell_\gamma\neq 0,\gamma\leq\alpha\leq i}}\{i-\gamma+p_\gamma-tf_\alpha\} ,
		\end{equation}
		with the convention {$p/0=\infty$, for $p\geq 0$}.
	\end{lemma}
	
	{Note that Assumption \ref{diag} intuitively ensures that contributions coming from mutants of different traits are different (when computing the growth or decrease rate of a given trait). It is then clear which one wins in Equation \eqref{defmi}. If this assumption does not hold,  it could happen that prefactors (in front of powers of $\mu$) matter, 
		and we do not want to enter into such an analysis.  Mathematically, it ensures the matrix $M_\nsided$ in \eqref{def-matrix} to be diagonalisable for $\mu$ small enough, and thus to obtain explicit expressions for change of basis matrices in the proof below. }
	
	\begin{proof}
		Under assumption \eqref{diag} the matrix $M_\nsided$ in \eqref{def-matrix} is diagonalisable for $\mu$ small enough:  it can be checked that
		$M_\nsided=SDS^{-1}$ with 
		\begin{align}
			D&=((f_i-b_i\mu)\delta_{ij})_{0\leq i,j \leq L},\\
			S&=\left(
			\left(\frac\nsided\mu\right)^{L-i}
			\frac{\prod_{k=i+1}^{L} (f_j-f_k)+\mu(b_k-b_j)}{\prod_{\ell=i}^{L-1} b_\ell}
			\ind{i\geq j}
			\right)_{0 \leq i,j \leq L}=:\left(\mu^{i-L}C_{ij}\ind{i\geq j}\right)_{0 \leq i,j \leq L}, \\
			S^{-1}&=\left( 
			\left(\frac\mu \nsided\right)^{L-j}
			\frac{\prod_{\ell=j}^{L-1} b_\ell}{\prod_{k=j,k\neq i}^{L} (f_i-f_k)+\mu(b_k-b_i)}
			\ind{i\geq j}
			\right)_{0 \leq i,j \leq L}=:\left(\mu^{L-j}C'_{ij}\ind{i\geq j}\right)_{0 \leq i,j \leq L}.
		\end{align}
		The solution to the system \eqref{triangle-syst} can then be written in the form 
		\be
		y(t)=\exp(tM_\nsided)y(0)=S\exp(tD)S^{-1}y(0),
		\ee
		which reads in coordinates, for $i=0,\dots, L$, 
		\begin{align}
			y_i(t)&=\sum_{\alpha,\gamma=0}^L
			S_{i\alpha}\eee^{t(f_\alpha-b_\alpha\mu)}S^{-1}_{\alpha\gamma}\ell_\gamma\mu^{p_\gamma}
			=\sum_{\gamma:\ell_\gamma\neq0}\sum_{\gamma\leq\alpha\leq i}
			\mu^{i-\gamma+p_\gamma}\eee^{t(f_\alpha-b_\alpha\mu)}\cdot
			C_{i\alpha}C'_{\alpha\gamma}\ell_\gamma.
		\end{align}
		Thus
		\begin{align}\label{pre-pow-mu}
			y_i(t\log(1/\mu))&=\sum_{\gamma:\ell_\gamma\neq0}\sum_{\gamma\leq\alpha\leq i}
			\mu^{i-\gamma+p_\gamma-t(f_\alpha-b_\alpha\mu)}\cdot
			\left(C_{i\alpha}C'_{\alpha\gamma}\ell_\gamma\right).
		\end{align}
		As $\mu$ tends to zero, the sum is dominated by the  term with the smallest exponent of $\mu$, which by
		definition is  $m_i(t)$, defined in \eqref{defmi}.
		Thus there exists a constant $C>0$, such that 
		\begin{equation}\label{pow-mu}
			y_i(t\log(1/\mu))
			=C \mu^{m_i(t)}\left(1+o(1)\right),
		\end{equation}
		which implies the assertion of \eqref{defmi} and concludes the proof of the lemma.
	\end{proof}

	\subsection{Before the swap}
	\subsubsection{Time interval $0\leq t\leq \Tinit_{L-1}$}
	{If $m_{ij}=m^{(2)}_{ij}$ let $\zeta=2$ and }
	\begin{align}\nonumber
		\tau^-_{L-1}(\ep,\mu)&=
		\inf\{t:\exists \ i\in\Ninterval{0}{L} \text{ s.t.\ } \xmu_i(t)>\mu^{i-\ep}\}\wedge
		\inf\{t:|\xmu_0(t)-\bar x_0|>\ep \}
		\wedge
		\inf\{t:\xmu_L(t)>\mu^{L-2+\ep}\},\\
	\end{align}
	{ while if  $m_{ij}=m^{(1)}_{ij}$ let $\zeta=1$ and}
	\begin{align}\nonumber
		\tau^-_{L-1}(\ep,\mu)&=
		\inf\{t:\exists \ i\in\Ninterval{0}{L} \text{ s.t.\ } \xmu_i(t)>\mu^{i-\ep}\}
		\wedge
		\inf\{t:|\xmu_0(t)-\bar x_0|>\ep \}\wedge
		\inf\{t:\xmu_L(t)>\ep\}.\\
	\end{align}
	{
		and define
		\begin{equation}
			T^-_{L-1}:=\lim_{\ep\to0}\lim_{\mu\to0}\frac{\tau^-_{L-1}(\ep,\mu)}{\log(1/\mu)}.
		\end{equation}
	}
	There exists a finite $C$ such that on the time interval $[0,\tau^-_{L-1}(\ep,\mu)]$, {for $i\in\llbracket 0,L\rrbracket$,}
	\begin{equation}
		\frac{d\xmu_i}{dt}\geq (f_{i0}-C\ep)\xmu_i+\mu\left(\frac{b_{i-1}}{\nsided}\xmu_{i-1}-b_i\xmu_i\right).
	\end{equation}
	Hence, by the Gronwall lemma, %and notations \eqref{triangle-syst}, 
	$\xmu$ is  bigger than the solution to $\frac{dy}{dt}= M_\nsided y$ with $f_i=f_{i0}-C\ep$. Applying Lemma \ref{lemma} with 
	$y(0)=(\xbar_0,0,\ldots,0)$
	and thus $\ell_0=\xbar_0$, $p_0=0$, $\ell_i=0$ for $i\neq0$, we get, using \eqref{A1}, for  $\ep$ small enough {and $t>0$,}
	\begin{equation}
		\lim_{\mu\to0}\frac{\log(\xmu_i(t\log(1/\mu)))}{\log(1/\mu)}\geq \begin{cases}
			-i	{	-C\eps t,	}		&\text{ for }i=0,\ldots,L-1,\\
			-L+t(f_{L0}-C\ep),	&\text{ for }i=L.
		\end{cases}
	\end{equation}
	On the other hand,  on the same time interval, we have, for some positive $C$, the upper bound
	\begin{equation}
		\frac{d\xmu_i}{dt}\leq (f_{i0}+C\ep)\xmu_i+\mu\left(\frac{b_{i-1}}{{\nsided}}\xmu_{i-1}-b_i\xmu_i\right)+E_i,
	\end{equation}
	where,  until  $\tau^-_{L-1}$, with  $\kappa:= \sup b_i/2$, 
	\be
	(E,\nsided)=\begin{cases}\left(\kappa\mu\cdot(\mu^{1-\ep},\mu^{2-\ep},\ldots,{\mu^{L-1-\ep}},\mu^{L-2+\ep},0),2\right)&,\hbox{\rm if }\;  m_{ij}=m^{(2)}_{ij}, \\
		\left((0,0,\ldots,0),1\right),  \hbox{\rm  if } m_{ij}=m^{(1)}_{ij}.
	\end{cases}
	\ee
	Again by  the Gronwall lemma, 
	$\xmu$ is  smaller than the solution to $\frac{dy}{dt}= {M_\nsided} y+E$, where the $f_i$ in $M_{\nsided}$ are given by $f_i=f_{i0}+C\ep$. 
	The variation of parameters method yields
	\begin{align}\nonumber
		y(t)&=\eee^{tM_\nsided}\left(y(0)+\left(\int_{0}^{t}\eee^{-sM_\nsided}ds\right) E\right)\\\nonumber
		&=\eee^{tM_\nsided}y(0)+S\left(\int_{0}^{t}\eee^{(t-s)D}ds\right)S^{-1} E\\
		&= \eee^{tM_\nsided}y(0)
		+S\left(\frac{\eee^{(f_{i}-b_i\mu)t}}{f_i-b_i\mu}\delta_{ij}\right)S^{-1} E
		-S\left(\frac{1}{f_i-b_i\mu}\delta_{ij}\right)S^{-1} E \label{err}.
	\end{align}
	Now we compute the order of magnitude of each term as in \eqref{pre-pow-mu} in the proof of Lemma 
	\ref{lemma} and show that the two terms in \eqref{err}  involving $E$ are negligible with respect to the 
	main term. Set
	\begin{equation}
		e_1(t):=S\left(\eee^{(f_{i}-b_i\mu)t}\delta_{ij}\right)S^{-1} E,\qquad
		e_2:=S\left(\delta_{ij}\right)S^{-1} E.
	\end{equation}
	In the case $m_{ij}=m^{(2)}_{ij}$ we have, for $i \neq L$, from Lemma \ref{lemma} that
	\begin{align}\nonumber
		\left(e_1(t\log(1/\mu))\vee e_2\right)_i
		&=O(\mu^{\min_{\gamma\in\llbracket 0,L-1\rrbracket,\gamma\leq i}\{i-\gamma+ (2+\gamma-\ep)
			\ind{\gamma<L-1}+\{i-\gamma+(\gamma+\ep)\}\ind{\gamma=L-1}\}})\\
		&=O(\mu^{(i+\ep)\ind{i=L-1}+(i+2-\ep)\ind{i<L-1}})=o(\mu^i),
	\end{align}
	and if $i=L$ we get 
	\begin{align}
		\left(e_1(t\log(1/\mu))\vee e_2\right)_L
		=O(\mu^{L-t(f_{L0}+C\ep)+\ep})=o(\mu^{L-t(f_{L0}+C\ep)}).
	\end{align}
	Consequently, proceeding as for the lower bounding ODE, we get
	\begin{align}\nonumber
		\lim_{\mu\to0}\frac{\log(\xmu_i(t\log(1/\mu)))}{\log(1/\mu)}
		&\leq
		\lim_{\mu\to0}\frac{\log(y_i(t\log(1/\mu)))}{\log(1/\mu)}\\
		&=\begin{cases}
			-i		{+C\eps t}	,				&\text{ for }i=0,\ldots,L-1,\\
			-L+t(f_{L0}+C\ep),			&\text{ for }i=L.
		\end{cases}
	\end{align}
	Finally observe that, as the only growing population is the one with trait $L$,
	\begin{equation}
		T^-_{L-1}=\lim_{\ep\to0}\lim_{\mu\to0}\frac{\tau^-_{L-1}(\ep,\mu)}{\log(1/\mu)}
		=\begin{cases}
			2/{f_{L0}},				&\text{ for }m_{ij}=m^{(2)}_{ij},\\
			L/{f_{L0}},			&\text{ for }m_{ij}=m^{(1)}_{ij}.
		\end{cases}
	\end{equation}
	In the case $m_{ij}=m^{(1)}_{ij}$ the proof continues directly with  Subsection \ref{sec-swap}.

	\subsubsection{Time interval ${\Tinit_{L-1}}\leq t\leq {\Tinit_{L-2}}$}
	Let $m_{ij}=m^{(2)}_{ij}$ and
	\begin{align}
		\tau^-_{L-2}(\ep,\mu)&=
		\inf\{t:\exists i\in\Ninterval{0}{L-1} \text{ s.t. } \xmu_i(t)>\mu^{i-\ep}\}\\\nonumber
		&\wedge
		\inf\{t:|\xmu_0(t)-\bar x_0|>\ep \}
		\wedge
		\inf\{t:\xmu_{L-1}(t)>\mu^{L-3+\ep}\}
		\wedge
		\inf\{t:\xmu_L(t)>\mu^{L-4+\ep}\}.
	\end{align}
	{and define 
		\begin{equation}
			\Tinit_{L-2}:=\lim_{\ep\to0}\lim_{\mu\to0}\frac{\tau^-_{L-2}(\ep,\mu)}{\log(1/\mu)}.
	\end{equation}}
	There exists a positive $C$ such that on the time interval $[\tau^-_{L-1}(\eps,\mu),\tau^-_{L-2}(\eps,\mu)]$,
	\begin{equation}
		\frac{d\xmu_i}{dt}\geq 
		(f_{i0}-C\ep)\xmu_i+
		\mu\left(
		\frac{b_{i-1}}{2}\xmu_{i-1}\ind{i<L-1}
		+\frac{b_{i+1}}{2}\xmu_{i+1}\ind{i=L-1}
		-b_i\xmu_i
		\right).
	\end{equation}
	
	Hence by the Gronwall lemma, and notations \ref{triangle-syst}, $\xmu$ is bigger than the solution to 
	\begin{equation}
		\frac{dy}{dt}=
		\begin{pmatrix}
			M_{\mathrm{left}}(L-2) & 0\\
			0		 & M_{\mathrm{right}} (1)
		\end{pmatrix}
		y=:M'(L-2,1)y
	\end{equation}
	where $M_{\mathrm{left}}(L-2)=M_2(L-2)$ with $f_i=f_{i0}-C\ep$ and 
	\begin{equation}
		M_{\mathrm{right}}=\begin{pmatrix}
			f_{L-1,0}-C\ep & \frac\mu2 b_L\\
			0			   & f_{L0}-C\ep
		\end{pmatrix}.
	\end{equation}
	Applying then twice Lemma \ref{lemma}, once with $M_{\mathrm{left}}(L-2)$ and $y_{\mathrm{left}}=(y_0,\ldots,y_{L-2})$ and once with $M_{\mathrm{right}}(1)$ 
	(treated as $M(1)$ with ``{reversed indices}",  i.e. $f_i,b_i$ replaced by $f_{L-i},b_{L-i}$) and $y_{\mathrm{right}}=(y_{L-1},y_L)$, with 
	\begin{equation}
		y(0)=(\xbar_0,\mu,\mu^2,\ldots,\mu^{L-1},\mu^{L-2}),
	\end{equation}
	{up to $o_\ep(1)$ terms in the powers of $\mu$ due to the range of possible initial conditions coming from the previous phase (those however do not change anything to the calculations),}
	we get
	\begin{equation}\label{LB1}
		\lim_{\mu\to0}\frac{\log(\xmu_i(t\log(1/\mu)))}{\log(1/\mu)}\geq \begin{cases}
			-i			{-C\eps t},				&\text{ for }i=0,\ldots,L-2,\\
			-(L-1)+t(f_{L0}-C\ep),	&\text{ for }i=L-1,\\
			-(L-2)+t(f_{L0}-C\ep),	&\text{ for }i=L.
		\end{cases}
	\end{equation}
	On the other hand, we have the upper bound
	\begin{equation}
		\frac{d\xmu_i}{dt}\leq 
		(f_{i0}+C\ep)\xmu_i+
		\mu\left(
		\frac{b_{i-1}}{2}\xmu_{i-1}\ind{i<L-1}
		+\frac{b_{i+1}}{2}\xmu_{i+1}\ind{i=L-1}
		-b_i\xmu_i
		\right)+E_i
	\end{equation}
	where until $\tau^-_{L-2}$ we have
	\begin{equation}
		E=\mu\cdot(\mu^{1-\ep},\mu^{2-\ep},\ldots,\mu^{L-2-\ep},\mu^{L-3+\ep},
		\mu^{L-2-\ep},\mu^{L-3+\ep}).
	\end{equation} 
	By the Gronwall lemma, $\xmu$ is smaller than the solution to
	$\frac{dy}{dt}= M' y+E$ with $f_i=f_{i0}+C\ep$. Using the same method as above (variation of  
	constants
	in the two blocks), we get \eqref{LB1} also as an upper bound,  with $f_{L0}-C\ep$ replaced by $f_{L0}+C\ep$, {and $-i-C\eps t$ replaced by $-i+C\eps t$}. 
	Finally observe that, 
	\begin{equation}
		\Tinit_{L-2}=\lim_{\ep\to0}\lim_{\mu\to0}\frac{\tau^-_{L-2}(\ep,\mu)}{\log(1/\mu)}=\frac4{f_{L0}}.
	\end{equation}

	\subsubsection{Induction until ${\Tinit_{ L/2}}$}
	{In this section if $L$ is odd, then $L/2$ has to be replaced by $\lfloor L/2\rfloor$.}
	For $k\in\{3,\ldots,L/2\}$ we treat the time interval $\Tinit_{L-k}\leq t\leq \Tinit_{L-(k+1)}$. Let $m_{ij}=m^{(2)}_{ij}$ and
	\begin{align}\nonumber
		\tau^-_{L-k}(\ep,\mu)&=
		\inf\{t:\exists i\in\Ninterval{0}{L-k+1} \text{ st } \xmu_i(t)>\mu^{i-\ep}\}\\
		&\wedge
		\inf\{t:|\xmu_0(t)-\bar x_0|>\ep \}
		\wedge
		\inf\{t: \exists j\in\Ninterval{1}{k}\text{ st } \xmu_{L-k+j}(t)>\mu^{(L-k+j)-2j+\ep}\}.
	\end{align}
	{ and define 
		\begin{equation}
			\Tinit_{L-k}:=\lim_{\ep\to0}\lim_{\mu\to0}\frac{\tau^-_{L-k}(\ep,\mu)}{\log(1/\mu)}.
		\end{equation}
	}
	{For $t\in[\tau^-_{L-k}(\ep,\mu),\tau^-_{L-k-1}(\ep,\mu)]$ } we have the lower bound 
	\begin{equation}
		\frac{d\xmu_i}{dt}\geq 
		(f_{i0}-C\ep)\xmu_i+
		\mu\left(
		\frac{b_{i-1}}{2}\xmu_{i-1}\ind{i<L-k+1}
		+\frac{b_{i+1}}{2}\xmu_{i+1}\ind{i\geq L-k+1}
		-b_i\xmu_i
		\right).
	\end{equation}
	Hence, by the Gronwall lemma, $\xmu$ is bigger than the solution to 
	\begin{equation}
		\frac{dy}{dt}=
		\begin{pmatrix}
			M_{\mathrm{left}}(L-k) & 0\\
			0		 & M_{\mathrm{right}} (k-1)
		\end{pmatrix}
		y=:M'(L-k,k-1)y,
	\end{equation}
	where $M_{\mathrm{left}}(L-k)=M_2(L-k)$ with $f_{i}=f_{i0}-C\ep$ and 
	\begin{equation}
		M_{\mathrm{right}}(k)=\begin{pmatrix}
			f_{L-k,0}-C\ep 	& \frac\mu2 b_{L-k+1}	& 			& 			0\\
			&\ddots					& \ddots 	&\\
			&						& f_{L-1,0}-C\ep	& \frac\mu2 b_{L}\\
			0			   	& 						&					& f_{L0}-C\ep
		\end{pmatrix}.
	\end{equation}
	
	Applying twice Lemma \ref{lemma}, once with $M_{\mathrm{left}}$ and $y_{\mathrm{left}}=(y_0,\ldots,y_{L-k-1})$ and once with $M_{\mathrm{right}}$ 
	(treated as $M(k)$ with ``{reversed indices}", $f_i,b_i$ replaced by $f_{L-i},b_{L-i}$) and $y_{\mathrm{right}}=(y_{L-k},\ldots,y_L)$, with
	\begin{equation}
		y(0)=(\xbar_0,\mu,\mu^2,\ldots,\mu^{L-k},\mu^{L-k+1},\mu^{L-k},\mu^{L-k-1},\ldots,\mu^{L-2k}),
	\end{equation}
	{up to $o_\ep(1)$ terms in the powers of $\mu$ due to the range of possible initial conditions coming from the previous phase,}
	we get
	\begin{equation}\label{LB2}
		\lim_{\mu\to0}\frac{\log(\xmu_i(t\log(1/\mu)))}{\log(1/\mu)}\geq \begin{cases}
			-i		{-C\eps t},					&\text{ for }i=0,\ldots,L-k,\\
			{-i+j-1+t(f_{L0}-C\ep)},	&\text{ for }i=L-k+j,j=1,\ldots,k.
		\end{cases}
	\end{equation}
	On the other hand, we have the upper bound
	\begin{equation}
		\frac{d\xmu_i}{dt}\leq 
		(f_{i0}+C\ep)\xmu_i+
		\mu\left(
		\frac{b_{i-1}}{2}\xmu_{i-1}\ind{i<L-k+1}
		+\frac{b_{i+1}}{2}\xmu_{i+1}\ind{i\geq L-k+1}
		-b_i\xmu_i
		\right)+E_i,
	\end{equation}
	where on the time interval $[\tau^-_{L-k}(\ep,\mu),\tau^-_{L-k-1}(\ep,\mu)]$ we have
	\begin{equation}
		E=\mu\cdot(\mu^{1-\ep},\mu^{2-\ep},\ldots,\mu^{L-k-\ep},\mu^{L-k-1+\ep},\mu^{L-k-\ep},\mu^{L-k-1+\ep},\mu^{L-k-2+\ep},\ldots,\mu^{L-1-2(k-1)+\ep}).
	\end{equation} 
	By the Gronwall lemma, $\xmu$ is thus smaller than the
	solution to $\frac{dy}{dt}= M' y+E$ with $f_i=f_{i0}+C\ep$. 
	Using the same method as above (variation of the constant in the two blocks), we get \eqref{LB2} 
	also as an upper bound, with $f_{L0}-C\ep$ replaced by $f_{L0}+C\ep$, {and $-i-C\eps t$ replaced by $-i+C\eps t$.}
	Finally, observe that
	\begin{equation}
		\Tinit_{L-k}=\lim_{\ep\to0}\lim_{\mu\to0}\frac{\tau^-_{L-k}(\ep,\mu)}{\log(1/\mu)}=\frac{2k}{f_{L0}}.
	\end{equation}
	
	\subsection{The swap}\label{sec-swap}
	
	Let $m_{ij}=m^{(2)}_{ij}$ and
	\begin{align}\nonumber
		\tau^s(\ep,\mu)&=
		\inf\{t:\exists i\in\Ninterval{0}{L/2} \text{ s.t. } \xmu_i(t)>\mu^{i-\ep}\}
		\wedge
		\inf\{t:\xmu_0(t)<\ep \}\\
		&\wedge
		\inf\{t:\exists i\in\Ninterval{L/2}{L} \text{ s.t. } \xmu_i(t)>\mu^{L-i-\ep}\}
		\wedge
		\inf\{t:\xmu_L(t)>\xbar_L-\ep \},
	\end{align}
	
	or $m_{ij}=m^{(1)}_{ij}$ and
	\begin{align}\nonumber
		\tau^s(\ep,\mu)&=
		\inf\{t:\exists i\in\Ninterval{0}{L-1} \text{ s.t. } \xmu_i(t)>\mu^{i-\ep}\}\\
		&\wedge
		\inf\{t:\xmu_0(t)<\ep \}
		\wedge
		\inf\{t:\xmu_L(t)>\xbar_L-\ep \}.
	\end{align}
	
	{For $t\in [\tau^-_{L/2}(\ep,\mu),\tau^s(\ep,\mu)]$ and} $\chi \in \{0,L\}$ we have the lower bounds 
	\begin{align}\label{2dLV}
		\frac{d\xmu_\chi}{dt}&\geq
		\left(b_\chi-d_\chi-c_{\chi0}\xmu_0-c_{\chi L}\xmu_L-C\mu\right)x_\chi^\mu-C'\mu^{1-\ep} 
	\end{align}
	and the upper bounds 
	\begin{align}\label{2dLV-upper}
		\frac{d\xmu_\chi}{dt}&\leq
		\left(b_\chi-d_\chi-c_{\chi 0}\xmu_0-c_{\chi L}\xmu_L\right)x_\chi^\mu+C'\mu^{1-\ep} .
	\end{align}
	{Let {$(\tilde x_0,\tilde x_L)$} denote the solution of the unperturbed system, i.e. of 
		{\begin{align}\label{2dLV-unperturbed}
				\frac{dx_\chi}{dt}&=
				\left(b_\chi-d_\chi-c_{\chi 0}x_0-c_{\chi L}x_L\right)x_\chi.
			\end{align}
		}
		By \eqref{A1} and \eqref{A2}, we know that this system has a
		unique stable equilibrium $(0,\xbar_L)$. 
		Moreover, the time needed
		to enter an $\ep-$neighbourhood of this equilibrium from  initial conditions $(\xbar_0-\ep,\ep)$ 
		is of order $O(1)$. 
		Applying the Gronwall lemma {to the function $|x_0^\mu(t)-\tilde x_0(t)|^2+|x_L^\mu(t)-\tilde x_L(t)|^2$}, \eqref{2dLV} and \eqref{2dLV-upper} imply that on any compact time interval 
		$(\xmu_0,\xmu_L)\to {(\tilde x_0,\tilde x_L)}$ { as $\mu\downarrow0$}. Moreover, for all $\mu$ small enough, the 
		system $(\xmu_0,\xmu_L)$ has a stable equilibrium that converges to $(0,\xbar_L)$,
		as $\mu\downarrow0$. }
	
	%The populations $(\xmu_0,\xmu_L)$ are bounded above and below by the perturbations \eqref{2dLV} and \eqref{2dLV-upper} of a 2-species Lotka-Volterra system with initial conditions $(\xbar_0-\ep,\ep)$. As the solutions are continuous in the parameters, they converge, as $\mu\to0$ to the solution of the unperturbed system. 
	%By \eqref{A1} and \eqref{A2} it is well known that there exists a
	%unique stable equilibrium $(0,\xbar_L)$. Moreover, the time needed
	%to enter an $\ep-$neighbourhood of this equilibrium is of order
	%$O(1)$. 
	
	For the populations $(\xmu_1,\ldots,\xmu_{L-1})$, we have,
	{for $t\in [\tau^-_{ L/2}(\ep,\mu),\tau^s(\ep,\mu)]$, } the lower bounds 
	\begin{align}
		\frac{dx_i^\mu}{dt}
		&\geq	\left(b_i-d_i-c_{i0}\xbar_{0}-c_{iL}\xbar_L-C\ep\right)x_i^\mu+
		\mu\left(
		\frac{b_{i-1}}{2}\xmu_{i-1}\ind{i<L/2}
		+\frac{b_{i+1}}{2}\xmu_{i+1}\ind{i\geq L/2}
		-b_i\xmu_i
		\right),
	\end{align}
	in the case $m_{ij}=m^{(2)}_{ij}$ and 
	\begin{align}
		\frac{dx_i^\mu}{dt}
		&\geq	\left(b_i-d_i-c_{i0}\xbar_{0}-c_{iL}\xbar_L-C\ep\right)x_i^\mu+
		\mu\left(
		{b_{i-1}}\xmu_{i-1}
		-b_i\xmu_i
		\right),
	\end{align}
	in the case $m_{ij}=m^{(1)}_{ij}$.
	We have decoupled traits $0$ and $L$ from traits $1,\ldots,L-1$. 
	We still have to show that the functions {$(x_i^\mu, i\in\Ninterval{1}{L-1})$}  stay smaller than $\mu^{1-\ep}$.
	By the Gronwall lemma, the following {holds:}
	\begin{enumerate}
		\item in the case $m_{ij}=m^{(1)}_{ij}$, the solution $(\xmu_1,\ldots,\xmu_{L-1})$ is smaller than the solution to 
		\begin{equation}
			\frac{dy}{dt}=
			M_1(L-2)y
		\end{equation}
		with $f_i=b_i-d_i-c_{i0}\xbar_{0}-c_{iL}\xbar_L-C\ep<0$ and initial conditions:
		\begin{equation}
			y(0)=(\mu,\mu^2,\ldots,\mu^{L-1}).
		\end{equation}
		{up to $o_\ep(1)$ terms in the powers of $\mu$ due to the range of possible initial conditions coming from the previous phase.}
		Applying Lemma \ref{lemma}, we get
		\begin{equation}
			\lim_{\mu\to0}\frac{\log(\xmu_i(t\log(1/\mu)))}{\log(1/\mu)}\geq 
			-i {+ t \sup_{1\leq \alpha \leq i}f_\alpha}, \quad	\text{ for }i=1,\ldots,L-1.
		\end{equation}
		But we just mentioned that the swap has a duration of order $1$. Thus, the $t$ to be considered is negligible with respect to $1$, and
		\begin{equation}
			\lim_{t\to0}\lim_{\mu\to0}\frac{\log(\xmu_i(t\log(1/\mu)))}{\log(1/\mu)}\geq 
			-i, \quad	\text{ for }i=1,\ldots,L-1.
		\end{equation}
		\item in the case $m_{ij}=m^{(2)}_{ij}$, the solution $(\xmu_1,\ldots,\xmu_{L-1})$ is thus smaller than the solution to 
		\begin{equation}
			\frac{dy}{dt}=
			\begin{pmatrix}
				M_{\mathrm{left}}(L/2-2) & 0\\
				0		 & M_{\mathrm{right}} (L/2-1)
			\end{pmatrix}
			y=:M'(L/2-2,L/2-1)y
		\end{equation}
		with $f_i=b_i-d_i-c_{i0}\xbar_{0}-c_{iL}\xbar_L-C\ep<0$ and initial conditions
		\begin{equation}
			y(0)=(\mu,\mu^2,\ldots,\mu^{L/2+1},\mu^{L/2},\mu^{L/2-1}\ldots,\mu)
		\end{equation}
		{up to $o_\ep(1)$ terms in the powers of $\mu$ due to the range of possible initial conditions coming from the previous phase.}
		{Here, if $L$ is odd, then the initial condition has to be replaced by
			\begin{equation}
				y(0)=(\mu,\mu^2,\ldots,\mu^{\lfloor L/2\rfloor},\mu^{\lfloor L/2\rfloor},\mu^{\lfloor L/2\rfloor-1}\ldots,\mu).
			\end{equation}	
			and matrix dimensions have to be modified accordingly, but the proof stays the same.
		}
		Applying  Lemma \ref{lemma} twice (in the two blocks), and letting $t$ go to $0$ as the swap has a duration of order $1$, we get
		\begin{equation}
			{\lim_{t \to 0}}\lim_{\mu\to0}\frac{\log(\xmu_i(t\log(1/\mu)))}{\log(1/\mu)}\geq 
			-\min\{i,{L-i}\},	\quad	\text{ for }i=1,\ldots,L-1.
		\end{equation}
	\end{enumerate}

	On the other hand, we have the  upper bound
	\begin{equation}
		\frac{dx_i^\mu}{dt}\leq
		F\xmu_i+\mu^{1-\ep},
	\end{equation}
	with some $F>0$. Thus, by the Gronwall lemma, 
	\begin{equation}
		\xmu_i(t)\leq  \mu^{i-\ep},\text{ for }t<\frac{\ep}{F}\log(1/\mu) \text{  and for }i=1,\ldots L/2.
	\end{equation}
	and similarly for $i=L/2,\ldots,L-1$ (no population changes its order of magnitude of more than $\ep$ during any time of order $O(1)$).
	We deduce that, for $i=1,\ldots,L-1$,
	\begin{equation}
		\lim_{\mu\to0}\frac{\log(\xmu_i(t\log(1/\mu)))}{\log(1/\mu)}\leq 
		\begin{cases}
			-\min\{i,{L-i}\} & \text{ if } m_{ij}=m^{(2)}_{ij}\\
			-i & \text{ if } m_{ij}=m^{(1)}_{ij}.
		\end{cases}
	\end{equation}
	The duration of  the swap  vanishes (on the time scale $\log(1/\mu)$) in the limit $\mu\to0$. We thus have $\Tswap=\Tinit_{L/2}$.

	\subsection{After the swap}
	\subsubsection{Case $m_{ij}=m^{(1)}_{ij}$}
	Let
	\begin{equation}
		\tau^+(\ep,\mu)=
		\inf\{t:\exists \ {i\in\Ninterval{1}{L-1} \text{ s.t.\ } \xmu_i(t)>\mu^{i-2\ep}}\}
		\wedge
		\inf\{t:|\xmu_L(t)-\bar x_L|>\ep\}.
	\end{equation}
	{For $t\in [\tau^s(\ep,\mu),\tau^+(\ep,\mu)]$, } we have the lower bound 
	\begin{equation}
		\frac{d\xmu_i}{dt}\geq (f_{iL}-C\ep)\xmu_i+\mu\left({b_{i-1}}\xmu_{i-1}-b_i\xmu_i\right).
	\end{equation}
	Hence, by the Gronwall lemma, and notations \eqref{triangle-syst}, $\xmu$ is bigger than the solution to
	$\frac{dy}{dt}= M_1 y$ with $f_i=f_{iL}-C\ep$. Applying Lemma \ref{lemma} with 
	\begin{equation}
		y(0)=(\ep,\mu,\ldots,\mu^{L-1},\bar x_L-\ep),
	\end{equation}
	{up to $o_\ep(1)$ terms in the powers of $\mu$ due to the range of possible initial conditions coming from the previous phase}
	(and thus $p_i=i-L\delta_{i,L}$), we get, using Assumption \ref{ass.1},
	\begin{align}\nonumber
		\lim_{\mu\to0}\frac{\log(\xmu_i(t\log(1/\mu)))}{\log(1/\mu)}&\geq
		-\min_{\alpha\leq i}\{i-L\delta_{i,L}-t(f_{\alpha L}-C\ep)\}\\
		&=-i+L\delta_{i,L}-t\min_{\alpha\in\Ninterval{0}{i}} |f_{\alpha L}|{+tC\ep}.
	\end{align}
	In the same way we get the corresponding  upper bound with $f_{\alpha L}-C\ep$ replaced by $f_{\alpha L}+C\ep$.

	\subsubsection{Case $m_{ij}=m^{(2)}_{ij}$}\label{par-after-swap-2sided}
	In this phase the system cannot be approximated by a piece-wise block-triangular linear system anymore. Let us study the ODE followed by the rescaled process. 
	Let
	\begin{align}
		\tau^+(\ep,\mu)&=
		\inf\{t>\Tswap:\exists i\in\Ninterval{0}{L-1}: \xmu_i(t)>\ep\}
		\wedge
		\inf\{{t:|\xmu_L(t)-\bar x_L|>\ep} \}.
	\end{align}

	{For $t\in [\tau^s(\ep,\mu),\tau^+(\ep,\mu)]$ } we have the lower bound 
	\begin{equation}
		\frac{d\xmu_i}{dt}\geq 
		(f_{iL}-C\ep)\xmu_i+
		\mu\left(
		\frac{b_{i-1}}{2}\xmu_{i-1}
		+\frac{b_{i+1}}{2}\xmu_{i+1}
		-b_i\xmu_i
		\right).
	\end{equation}
	and a similar upper bound where $f_{iL}-C\ep$ is replaced by $f_{iL}+C\ep$.
	Let
	\begin{equation}
		\xtep_i:=	\frac{\log\left[x_i^{\mu}\left(t\cdot\log\left(\frac1\mu\right)\right)\right]}{\log(\frac1\mu)}.
	\end{equation}
	We thus have 
	\begin{equation}\label{eq-xtilde}
		\frac{d\xtep_i}{dt}\geq 
		f_{iL}-C\ep-\mu+
		\frac{b_{i-1}}{2}\mu^{1+\xtep_i-\xtep_{i-1}}
		+\frac{b_{i+1}}{2}\mu^{1+\xtep_i-\xtep_{i+1}}
	\end{equation}
	and a similar upper bound,
	with initial condition (we reset the time of the swap to 0 from now on):
	\begin{equation}\label{3rd_phase_ci}
		\xtep(0)=\left(\frac{\log\ep}{\log(1/\mu)},-1,-2,\ldots,-L/2,-L/2+1,\ldots,-1,\frac{\log(\bar{x}_L-\ep)}{\log(1/\mu)}\right).
	\end{equation}
	{up to $o_\ep(1)$ terms due to the range of possible initial conditions coming from the previous phase.}
	{Here, if $L$ is odd, then the initial condition has to be replaced by
		\begin{equation}
			\xtep(0)=\left(\frac{\log\ep}{\log(1/\mu)},-1,-2,\ldots,-\lfloor L/2\rfloor,-\lfloor L/2\rfloor,\ldots,-1,\frac{\log(\bar{x}_L-\ep)}{\log(1/\mu)}\right),
		\end{equation}	
		but the proof idea stays the same.
	}
	Let $\delta>0$ and 
	$T_i^{in}(\delta,\mu):=\inf\{t>0 : \xtep_i>\xtep_{i-1}-(1-\delta) \text{ or } \xtep_i>\xtep_{i+1}-(1-\delta) \}$.  
	Then, for $t\in[0,T_i^{in}]$, that is, when $\xtep_i$ is above one of its neighbours minus $1-\delta$, then, for $\mu$ and $\delta$ small enough, the slope of $\xtep_i$ is prescribed by the fitness of trait $i$ with respect to trait $L$ (up to a multiple of  $\ep$). Indeed, by \eqref{eq-xtilde}:
	\begin{align}\label{xtepLbounds}
		d\xtep_i/dt&\geq f_{iL}-C\ep-\mu\geq f_{iL}-2C\ep , \nonumber \\
		d\xtep_i/dt&\leq f_{iL}+C\ep-\mu+
		\frac{b_{i-1}}{2}\mu^{\delta}
		+\frac{b_{i+1}}{2}\mu^{\delta}\leq f_{iL}+2C\ep.
	\end{align}
	Let 
	$T_i^{out}(\delta,\mu):=\inf\{t>\Tswap : \xtep_i<\xtep_{i-1}-(1+\delta) \text{ or } \xtep_i<\xtep_{i+1}-(1+
	\delta) \}$, we call it the exit time of the security region. 
	Let us show that, for $\mu$ small enough,  we have $T_i^{out}=\infty$, for all $i\in\{0,\ldots,L\}$. \\
	Assume by contradiction that $\inf\{T_i^{out},i\in\{0,\ldots,L\}\}<\infty$. 
	
	Among the indices $i$ that reach the infimum, consider the one such that $\xtep_i(T_i^{out})$ is maximal, that is $\xtep_i$ is the highest population among those which exit the security region first. By continuity of the solutions, at $t=T_i^{out}$ we have
	$\min\{1+\xtep_i-\xtep_{i-1},1+ \xtep_i-\xtep_{i+1}\}=-\delta $. 
	Suppose that $\xtep_i$ exits its security region by falling below its left neighbour minus one, i.e:
	\begin{equation}\label{exit}
		1+\xtep_i-\xtep_{i-1}=-\delta\quad \text{ and }\quad  1+ \xtep_i-\xtep_{i+1}>-\delta,
	\end{equation}
	the two other possibilities (right neighbour or both) are similar. By \eqref{eq-xtilde}, for $\mu$ small enough:
	\begin{align}
		\frac{d\xtep_i}{dt}(T_i^{out})
		&\geq 
		f_{iL}-C\ep-\mu+
		\min\left\{\frac{b_{i-1}}2,\frac{b_{i+1}}2\right\}
		\mu^{-\delta}
		=O(\mu^{-\delta}),
		\label{i}
		\\
		\frac{d\xtep_{i-1}}{dt}(T_i^{out})
		&=
		f_{i-1,L}-C\ep-\mu+
		\frac{b_{i-2}}{2}\mu^{1+\xtep_{i-1}-\xtep_{i-2}}
		+\frac{b_{i}}{2}\mu^{2+\delta}
		\leq 
		\frac{b_{i-2}}{2}o(\mu^{-\delta})
		+\frac{b_{i}}{2}\mu^{2+\delta}
		=o(\mu^{-\delta}).
		\label{i-1}
	\end{align}
	where the upper bounds in \eqref{i-1} come from the assumption that $\xtep_i$ is the highest population
	{among these}
	exiting their security region. 
	Indeed, if in \eqref{i-1} we had $1+\xtep_{i-1}-\xtep_{i-2}={-\delta}$ then, by definition ,$\xtep_{i-1}$ would
	exit its security region, thus we would have 
	$\xtep_{i-1}\leq\xtep_i$, which contradicts \eqref{exit}. Thus, $1+\xtep_{i-1}-\xtep_{i-2}>{-\delta}$.
	The equations \eqref{i} and \eqref{i-1} imply that the derivative $\frac d{dt}(\xtep_i-\xtep_{i-1})(T_i^{out})$ is as large as needed.
	Thus again by the continuity of the solutions, this implies the existence of some $t'<t$ such that 
	$\xtep_i(t')<\xtep_{i-1}(t')-(1+\delta)$. Hence $t'< T_i^{out}$, which is a contradiction. 
	This implies that $T_i^{out}=\infty$, for all $i\in\{0,\ldots,L\}$. 
	
	This allows us to describe the limit of $\xtep$ as $\mu\to0$. 
	A helpful example is given on Figure \ref{fig-pow-general}. First, as $f_{LL}=0$, equation \eqref{xtepLbounds} and the previous reasoning imply that until $\tau^+(\ep,\mu)$:
	\begin{equation}
		\frac{d\xtep_L}{dt}\geq 
		-C\ep \quad \text{ and } \quad 
		\frac{d\xtep_L}{dt}\leq 
		+C\ep,
	\end{equation}
	which implies that $\xtep_L\to0$ (take the limits $\ep\to0$ after $\mu\to0$). Now the initial condition \eqref{3rd_phase_ci} and Assumption \ref{ass.1} imply that $\xtep_i\to-(L-i)$, for $i=L/2,\ldots, L-1$. Indeed, $\xtep_L$ is close to 0, and $\xtep_i(0)=-(L-i)$ for those indices, so the only possibility to maintain a difference of less than one with their nearest neighbours and having a negative fitness $f_{i,L}$ is to stay constant.
	The shape of the first $L/2$  coordinates of the process is less trivial to formulate: each $\xtep_i$ behaves piece-wise linearly in the limit $\mu\to0$ and given the sequence $(f_{0L}, \ldots, f_{L-1,L})$ one can construct the successive slopes by following the rule ``$\xtep_i$ tries to decay with slope $f_{iL}$ while being at distance at most 1 of $\xtep_{i-1}$ and $\xtep_{i+1}$;  if it is not possible then it stays parallel to the largest of its neighbours, either $\xtep_{i-1}$ or $\xtep_{i+1}$".
	
	More precisely, consider the sequence $\{i_1,\ldots,i_r\}$ of "fitness records" defined recursively by $i_1=0$, 
	$i_k=\min\{i\in\llbracket 0,L-1 \rrbracket : f_{iL} <f_{i_{k-1}L} \}$. Then the previous reasoning implies that, for any $\ep>0$, as $\mu\to0$, 
	the process $(\xtep(t))_{t>0}$, starting with initial condition \eqref{3rd_phase_ci}, stays in an $\ep$-neighbourhood of the deterministic process $x(t)$ given by:
	\begin{align}
		x_i(t)=
		-(L-i)
		\vee
		\smash{\displaystyle\max_{k\in\llbracket0, i \rrbracket }}	\{-i-|f_{kL}|t\}		
		\vee
		\smash{\displaystyle\max_{k\in\llbracket 1, r \rrbracket }}	\{-i_k-|i-i_k|-|f_{i_kL}|t\}.			
	\end{align}
	Once again, Figure \ref{fig-pow-general} provides a helpful example to compute the formula.\qed

	\section{Proofs of Theorems \ref{pro_phase1_mugrand} and  \ref{pro_mupetit}}
	
	In this section we focus on mutation probabilities scaling as a negative power of $K$ {times a slowly varying function (recall \eqref{muenpuissancedeK}).
		
		\subsection{Poisson representation}\label{sectionpoisson}

		In the vein of Fournier and M\'el\'eard \cite{FM04}, we represent the population process in terms 
		of Poisson measures. Let $(Q_k^{(b)},Q_k^{(m)},Q_k^{(d)}, 0 \leq k \leq L)$ be independent Poisson random measures on $\R_+^2$ with intensity $ds d\theta$, 
		and {recall that} $(e_i, 0 \leq  i \leq L)$ is the canonical 
		basis of $\R^{L+1} $.
		We decompose on possible jumps that may occur: births without mutation, birth with mutation and deaths of individuals.
		For simplicity, we write 
		\begin{equation}\label{deathrate}
			d^{K}_{i}(x) = D^{K}_{i}(x)x_i = \left( d_{i} + \sum_{j=0}^L \frac{c_{ij}}{K}x_j \right) x_i,
		\end{equation}
		for  the total death rate of {the} sub-population {$i$.} 
		Recall that in this regime, we only consider the mutation kernel $m_{ij}^{(1)}= \mu \delta_{i+1,j}$.
		The process $X$ admits the following representation.
		For every real-valued function $h$ on $\R_+^{L+1}$ such that $h(X(t))$ is integrable,
		\begin{align}
			\label{defN}\nonumber
			h(X(t))=h(X(0))&+ \sum_{k=0}^{L} \int_0^t\int_{\R_+}  \Big(  
			h(X({s^-})+e_k)-h(X({s^-}))\Big)
			\mathbf{1}_{\theta \leq (1- \mu)b_{k}  X_k(s^-)} 
			Q_k^{(b)}(ds,d\theta)\\\nonumber
			&+\sum_{k=1}^{L} \int_0^t\int_{\R_+}  \Big(  
			h(X({s^-})+e_k)-h(X({s^-}))\Big) \mathbf{1}_{\theta \leq \mu b_{k-1}  X_{k-1}(s^-) }Q_{k}^{(m)}(ds,d\theta)
			\\
			&+\sum_{k=0}^{L} \int_0^t\int_{\R_+}  \Big(  
			h(X({s^-})-e_k)-h(X({s^-}))\Big)
			\mathbf{1}_{\theta  \leq D^K_k(X(s^-))X_k(s^-)}
			Q_k^{(d)}(ds,d\theta).
		\end{align}
		
		Let us now introduce a finite subset of $\N$ containing 
		the equilibrium size of the $0$-population,
		\begin{equation} \label{compact1}I_\eps^K:= \Big[K\Big(\bar{x}_0-2\eps \frac{\sup_{1 \leq i \leq L}c_{0i}}
			{c_{00}}\Big),
			K\Big(\bar{x}_0+2\eps \frac{\sup_{1 \leq i \leq L}c_{0i}}{c_{00}}\Big)\Big]\cap \N, \end{equation}
		and the stopping times $T^K_\eps$ and $S^K_\eps$, which denote respectively the hitting time of $\lfloor\eps K \rfloor$ by the total mutant 
		population ($X_1+...+X_L$) and the exit time of $I_\eps^K$ by the resident $0$-population,
		\begin{equation} \label{TKTKeps1} T^K_\eps := \inf \Big\{ t \geq 0, \sum_{1 \leq i \leq L}X_i(t)= \lfloor \eps K \rfloor \Big\},
			\quad S^K_\eps := \inf \Big\{ t \geq 0, X_0(t)\notin I_\eps^K \Big\}.  \end{equation}

		As shown in  \cite{champagnat2006microscopic}, we know that as long as the total mutant population 
		size is smaller than $\eps K$, the resident population 
		size stays close to its monomorphic equilibrium with a probability close to $1$ (see Lemma 
		\ref{Th3cChamp}).
		{This is a fundamental property of the population process, as it implies that the populations live in an almost constant environment and 
			are subject to an almost constant competitive pressure from other individuals, $c_{i0}\bar{x}_0$. This  allows us to couple $i$-population sizes 
			($1 \leq i \leq L-1$) with subcritical branching processes with migration $X_i^{(-)}$ and $X_i^{(+)}$ to control their dynamics.}
		{Moreover, after the first growing phase for the $L$-population,}
		if the sum of the $1$- to $(L-1)$-mutant population sizes stays smaller than $\eps K$, whereas the 
		$L$-mutant population size exceeds the size 
		$\eps K$, the $0$ and $L$ populations behave as if they were the only ones in competition. As a 
		consequence, the remaining time needed for the $L$-population
		to replace the $0$-population is close to $\log K/|f_{0L}|$ (see,  for instance, 
		\cite{champagnat2006microscopic} and later in this paper for a precise statement).
		Hence, the first step consists in
		estimating the time needed for the mutant population to reach the size $\lfloor \eps K \rfloor$.
		There are essentially two possibilities:
		\begin{itemize}
			\item Either {$K\mu^{L} \gg 1$;} in this case there is a (large) number of order 
			$K\mu^{L}$ of $L$-type individuals. Hence the outcome is similar to a large resident population producing recurrently favourable 
			mutants, studied in details 
			in \cite{Smadi2017}. The fixation time of the trait $L$ is of order $\log K$, 
			and we provide couplings with appropriate birth-death processes (without competition) with immigration to 
			control the subpopulation sizes.
			\item Or $K\mu^{L} \ll 1$; in this case some of the mutant-population size dynamics consist 
			in small excursions separated 
			with periods with no individual. Indeed, the $i$-population with $i\leq\alpha$ is again well 
			approximated by a 
			birth-death {process} (without competition) with immigration, which is close to the deterministic limit, 
			while, for the $i$-population with $i>\alpha$, 
			the immigration term is not large enough and the population is well described, at each arrival of a single mutant, 
			by a subcritical birth-death process. Each {excursion of} the sum of populations $i\in\llbracket\alpha,L-1\rrbracket$ has the same probability to produce a
			$L$-mutant which may generate a large population and invade. In this case, 
			the time to invasion is close to a geometric random variable, with a mean of order $1/(K\mu^L)$, much larger than $\log K$.
		\end{itemize}

		\subsection{Proof of Theorem \ref{pro_phase1_mugrand}} \label{alpha+1geqM}
		
		The time needed for the favourable mutation to invade the population  depends strongly on the mutation 
		probability per reproductive event, $\mu$.
		
		To study the {case when $K\mu^L\gg1$,} we  couple each population size $X_i$, $0 \leq i \leq L-1$ with two processes such that, 
		for every $0\leq i \leq L-1$ and $t \leq T_\eps^K \wedge S_\eps^K$,
		\begin{equation} \label{coupling1}  X^{(-)}_i(t) \leq X_i(t) \leq X^{(+)}_i(t), \quad a.s. \end{equation}
		
		By definition of the population process in \eqref{defN} and of the stopping times $T_\eps^K$ and $S_\eps^K$ in \eqref{TKTKeps1}, the 
		following processes satisfy \eqref{coupling1}:
		
		\begin{equation}\label{defno+-} X^{(\pm)}_0(t)= K\Big(\bar{x}_0 \pm 2\eps \frac{\sup_{1 \leq i \leq L}c_{0i}}
			{c_{00}}\Big)=:
			x_0^{(\pm)}K, \end{equation}
		and, for $1 \leq i \leq L-1$ {and $* \in \{-,+\}$,
			\begin{align}\nonumber
				\label{defNi1}
				X^{(*)}_i(t)= &\int_0^t\int_{\R_+}  \mathbf{1}_{\theta \leq (1- \mu)b_{i} X^{(*)}_i(s^-)} 
				Q_k^{(b)}(ds,d\theta)+ \int_0^t\int_{\R_+}  \mathbf{1}_{\theta \leq \mu b_{i-1}  X^{(*)}_{i-1}(s^-) }Q_{k}^{(m)}(ds,d\theta)
				\\ &- \int_0^t\int_{\R_+} \mathbf{1}_{\theta  \leq (d_i+ c_{i0}x_0^{(\bar{*})} 
					+\mathbf{1}_{\{*=-\}}\eps \sup_{1 \leq j \leq L}c_{ij}) X^{(*)}_i(s^-)}
				Q_k^{(d)}(ds,d\theta),
			\end{align}
			where $\bar{*}= \{-,+\} \setminus *$ and} we used the same Poisson measures as in \eqref{defN}.
		Note that from this representation, we get directly the classical semi-martingale decomposition for $X^{(-)}_i$ and $X^{(+)}_i$:
		for $* \in \{-,+\}$,
		\be
		X^{(*)}_i(t)= M^{(*)}_i(t)+A^{(*)}_i(t), 
		\ee
		where $M_i^{(*)}$ is a square integrable martingale and $A^{(*)}_i$ is a finite variation process, namely
		\begin{align}\nonumber
			M_i^{(-)}(t)=&\int_0^t\int_{\R_+}  \mathbf{1}_{\theta \leq (1- \mu)b_{i} X^{(-)}_i(s^-)} 
			(Q_k^{(b)}(ds,d\theta)-dsd\theta)+ \int_0^t\int_{\R_+}  \mathbf{1}_{\theta \leq \mu b_{i-1}  X^{(-)}_{i-1}(s^-) }(Q_{k}^{(m)}(ds,d\theta)-dsd\theta)
			\\ &- \int_0^t\int_{\R_+} \mathbf{1}_{\theta  \leq (d_i+ c_{i0}x_0^{(+)}
				+\eps \sup_{1 \leq j \leq L}c_{ij}) X^{(-)}_i(s^-)}
			(Q_k^{(d)}(ds,d\theta)-dsd\theta),
		\end{align}
		\begin{equation} \label{Ai}
			A_i^{(-)}(t)= \mu b_{i-1}\int_0^t  X^{(-)}_{i-1}(s)ds +\left((1-\mu)b_i-d_i-c_{i0}x_0^{(+)} 
			-\eps \sup_{1 \leq j \leq L}c_{ij}\right) \int_0^t X^{(-)}_i(s)ds,
		\end{equation}
		and the same expression for $M^{(+)}_i$ and $A^{(+)}_i$ by replacing the $(-)$'s by $(+)$'s and the terms
		\be
		d_i+c_{i0}x_0^{(+)} 
		+\eps \sup_{1 \leq j \leq L}c_{ij} 
		\ee
		by 
		\be
		d_i+c_{i0}x_0^{(-)}.
		\ee
		Finally,  we recall the expression of the quadratic variation of $M_i^{(-)}$,
		\begin{equation}
			\label{defqvNi1}
			\langle M^{(-)}_i\rangle_t=  \mu b_{i-1}  \int_0^t X^{(-)}_{i-1}(s)ds+
			\left((1- \mu)b_{i}+d_i+ c_{i0}x_0^{(+)} 
			+\eps \sup_{1 \leq j \leq L}c_{ij}\right) \int_0^t X^{(-)}_i(s)ds,
		\end{equation}
		and
		the one of $M_i^{(+)}$ is obtained by similar modifications {as} before.\\

		{Let us now introduce, for $1 \leq k \leq L-1$, the following notations:}
		\begin{equation}\label{defspm}
			-s_{k0}^{(+)}:=(1-\mu)b_k-d_k- c_{k0}x_0^{(-)}, \quad \text{and} \quad 
			-s_{k0}^{(-)}:=(1-\mu)b_k-d_k- c_{k0}x_0^{(+)}-\eps \sup_{1 \leq j \leq L}c_{kj},
		\end{equation}
		as well as, for $* \in \{-,+\}$:
		{\begin{equation}\label{deftepsi}
				x_k^{(*)}:=(1*\eps)^{k}\frac{b_0... b_{k-1} x_0^{(*)}\mu^k}{{s_{10}^{(*)}...s_{k0}^{(*)}}} \quad \text{and}
				\quad t^{(k)}_\eps := \frac{|\ln \eps|}{s_{k0}^{(-)}} .
			\end{equation}
			Notice that  $s_{k0}^{(+)}\leq s_{k0}^{(-)}$ and that $s_{k0}^{(+)}$ and $s_{k0}^{(-)}$ are positive,
			for $\eps$ small enough, 
			by Assumption (A2).
			
			\begin{lemma} \label{lemme_bounds_ni} 
				For every $0 \leq i \leq L-1$,
				$$
				x_i^{(-)}K\leq \E\left[X_i^{(-)}(s)\right], \quad  s\geq t_\eps^{(1)}+...+t_\eps^{(i)} , 
				$$
				and
				$$ \E \left[X_i^{(+)}(s) \right]\leq x_i^{(+)}K, \quad  s\geq 0. 
				$$
			\end{lemma}

			\begin{proof} 
				We prove this Lemma by induction.
				The property is true for $i=0$. Recall \eqref{defNi1}, \eqref{defspm} and \eqref{deftepsi}.
				Then we get for $1 \leq i \leq L-1$
				$$
				\frac{d}{dt}\E[X_i^{(*)}(t)]= b_1 \mu\E[X_{i-1}^{(*)}(t)]  -s_{i0}^{(*)} \E[X_i^{(*)}(t)], \quad \E[X_i^{(*)}(0)]=0.
				$$
				By the induction hypothesis, this yields for every $t \geq 0$,
				$$ \frac{d}{dt}\E[X_i^{(+)}(t)]\leq \mu b_{i-1}  x_{i-1}^{(+)}K -s_{i0}^{(+)} \E[X_i^{(+)}(t)], \quad \E[X_i^{(*)}(0)]=0 $$
				and, for every $t \leq t^{(1)}_\eps$,
				$$ \frac{d}{dt}\E[X_i^{(-)}(t)]\geq \mu b_{i-1}  x_{i-1}^{(-)}K   -s_{i0}^{(-)} \E[X_i^{(-)}(t)], \quad \E[X_i^{(*)}(0)]=0, $$
				which ends the proof.
			\end{proof}
			
			\begin{lemma} \label{diff_couting_processes}
				Let $0 \leq i < L$ such that $\lim_{K \to \infty} K \mu^i = \infty$, and introduce the two counting processes:
				\begin{equation} \label{defRi}
					R^{(\pm)}_i(t):=\int_0^t \int_{\R^+} \mathbf{1}_{\theta \leq \mu b_{i}X^{(\pm)}_{i}(s^-)}
					Q_{i+1}^{(m)}(ds,d\theta),
				\end{equation}
				and
				\begin{equation}\label{defbarRi}
					\bar{R}^{(\pm)}_i(t):=\int_0^t \int_{\R^+} \mathbf{1}_{\theta \leq \mu b_{i}\E[X^{(\pm)}_{i}(s^-)]}
					Q_{i+1}^{(m)}(ds,d\theta),
				\end{equation}
				where we use the same Poisson point measure as in \eqref{defN}. Then 
				$ M^{(\pm)}_i:=R^{(\pm)}_i-\bar{R}^{(\pm)}_i $
				is a martingale and
				$$ \E\left[ (M^{(\pm)}_i(t))^2 \right] \leq 2 \mu b_i x_i^{(\pm)}Kt. $$
			\end{lemma}
			
			\begin{proof} 
				We have
				\begin{align*}
					M_i^{(\pm)}(t) &= \int_0^t \int_{\R^+} \left(\mathbf{1}_{\theta \leq \mu b_{i}X_{i}^{(\pm)}(s^-)}-
					\mathbf{1}_{\theta \leq \mu b_{i}\E[X^{(\pm)}_{i}(s^-)]}\right)
					Q_{i+1}^{(m)}(ds,d\theta)\\
					&= \int_0^t \int_{\R^+} \left(\mathbf{1}_{\theta \leq \mu b_{i}X^{(\pm)}_{i}(s^-)}-
					\mathbf{1}_{\theta \leq \mu b_{i}\E[X^{(\pm)}_{i}(s^-)]}\right)
					\tilde{Q}_{i+1}^{(m)}(ds,d\theta) + \mu b_{i} \int_0^t \left( X^{(\pm)}_i(s)- \E[X^{(\pm)}_i(s)] \right)ds.
				\end{align*}
				Hence, $M^{(\pm)}_i$ is a martingale. We can compute its quadratic variation via
				\begin{align*}
					\langle M^{(\pm)}_i\rangle_t &= \int_0^t \int_{\R^+} \left(\mathbf{1}_{\theta \leq \mu b_{i}X^{(\pm)}_{i}(s)}-
					\mathbf{1}_{\theta \leq \mu b_{i}\E[X^{(\pm)}_{i}(s)]}\right)^2ds d\theta\\
					&= \mu b_{i} \int_0^t  \left(X^{(\pm)}_{i}(s)+\E[X^{(\pm)}_{i}(s)]- 2 \left( X^{(\pm)}_{i}(s)\wedge\E[X^{(\pm)}_{i}(s)] \right) \right)ds 
					= \mu b_{i} \int_0^t  \left|X^{(\pm)}_{i}(s)-\E[X^{(\pm)}_{i}(s)]\right|ds 
				\end{align*}
				As a consequence
				$$ \E\left[ (M^{(\pm)}_i(t))^2 \right] = \E\left[ \langle M^{(\pm)}_i\rangle_t \right]  
				\leq  \mu b_{i} \int_0^t  \E\left[X^{(\pm)}_{i}(s)+\E[X^{(\pm)}_{i}(s)]\right]ds
				= 2 \mu b_{i} \int_0^t  \E\left[X^{(\pm)}_{i}(s)\right]ds,$$
				and we end the proof applying Lemma \ref{lemme_bounds_ni}.
		\end{proof}}
		
		We have now the tools needed to prove Theorem \ref{pro_phase1_mugrand}.
		
		\begin{proof}[Proof of Theorem \ref{pro_phase1_mugrand}]
			From \eqref{coupling1} and Lemma \ref{lemme_bounds_ni} we know that the $L$-population has a size of order {$K\mu^L$} after a time of order
			$\ln(1/\eps)$, for $\eps$ small enough (not scaling with $K$). The proof of the asymptotics
			$$ \liminf_{K \to \infty} \P \left( (1-c\eps)\frac{1}{\alpha} \frac{L}{f_{L0}}< 
			\frac{T^{(K,L)}_\eps }{\log K} < (1+c\eps)\frac{1}{\alpha} \frac{L}{f_{L0}} \right) \geq 1-c\eps
			$$
			follows this of Lemma 6.1 in \cite{Smadi2017}.
			{To end the proof of Theorem \ref{pro_phase1_mugrand}}, two more steps are needed. 
			{The first one is the study of the swap between 0 and $L$-populations, which 
				leads to the first statement \eqref{eq1th} of Theorem \ref{pro_phase1_mugrand}, 
				and the second one is the study of the extinction phase of the unfit mutants, which  
				leads to the second and third statements \eqref{eq2th} and \eqref{eq3th} of Theorem 
				\ref{pro_phase1_mugrand}.}

			First we need to show that once the $L$-population size 
			has reached the value $\eps K$, the rescaled populations $X_0^K$ and $X_L^K$ behave as if they were the only ones in competition and follow 
			a dynamics close to the solutions to \eqref{cde-gen}
			with $L=1$, $\mu=0$ and initial conditions satisfying
			\be
			x_L(0)=\eps \quad \text{and} \quad |x_0(0)-\bar{x}_0|\leq 2\eps \frac{\sup_{1 \leq i \leq L}c_{0i}}{c_{00}}. 
			\ee
			This stays true until a time when $X_L^K$ is close to its monomorphic equilibrium size $\bar{x}_L$ and $X_0^K$ is smaller than $\eps^2$. 
			During this time interval, the $i$-population sizes, for $1\leq i \leq L-1$,
			do not evolve a lot.
			More precisely, there exists a constant $\eps_0$ such that, for $\eps \leq \eps_0$ and 
			$1 \leq i \leq L-1$, with a probability close to one,
			$ {\mu^{i+\eps} \leq X_i^K(t) \leq \mu^{i-\eps}}$,
			where $t$ describes an interval with a duration of order $1$, which is the time needed for the rescaled 
			population sizes 
			$(X_0^K,X_L^K)$ to hit 
			the set $((0,\eps^2],[\bar{x}_L-\eps,\bar{x}_L+\eps])$ from an initial state close to $(\bar{x}_0,\eps)$.
			To prove that, the idea is to show that the total population size stays of order $K$, and as a consequence with a probability 
			close to one, 
			we can find a positive $A$ such that $-A \E[X_i^K(t)] \leq d \E[X_i^K(t)]/dt \leq A \E[X_i^K(t)]$
			(for rigorous arguments, see the proof of Lemma 10 in \cite{billiard2017interplay}). This leads to the 
			following rigorous statement: there exist a 
			positive $\eps_0$ and a function $f: x \mapsto f(x) \in (0,x^2)$ such that, for $\eps \leq \eps_0$, 
			there exist a stopping time $U_\eps^K$ and an event $\mathcal{E}$ such that
			\begin{equation} \label{firstpropE} \frac{U_\eps^K f_{L0}}{\log K}\underset{K \to \infty}{\to} 1, \quad \text{in probability}, \qquad
				\P(\mathcal{E})\geq 1-\eps, \end{equation}
			and almost surely on $\mathcal{E}$,
			\begin{equation} \label{onmathcalE} f(\eps)< {X^K_0(U_\eps^K)}<\eps^2, \ 
				\left|{X^K_L(U_\eps^K)}-\bar{x}_L\right|\leq \eps, \ {K\mu^{\eps}} < {X^K_i(U_\eps^K)/\mu^i}<K{\mu^{-\eps}}, 
				\ 1\leq i \leq L-1. \end{equation}
			{This proves part \eqref{eq1th} of Theorem \ref{pro_phase1_mugrand}.}
			
			Second, we need to approximate the time for the  $i$-populations ($0 \leq i \leq L-1$) to get extinct after the time $U_\eps^K$.
			Let us define two stopping times
			$$ V_\eps^K:= \inf \left\{ t \geq U_\eps^K, |{X^K_L(t)}-\bar{x}_L|>2\eps \right\}, 
			$$
			and
			$$
			W_\eps^K:= \inf \Big\{ t \geq U_\eps^K, \sum_{0\leq i \leq L-1} {X^K_i(t)}>\eps \Big\}. 
			$$
			We will prove the following property: there exist $\eps_0,C,V>0$ such that, for $\eps \leq \eps_0$,
			{\begin{equation}\label{eqpopM}
					\liminf_{K \to \infty} \P\left( \eee^{KV} < \left( W_\eps^K \wedge  V_\eps^K \right)\right)\geq 1-o_\eps(1),
				\end{equation}
				where $o_\eps(1)$ is a function of $\eps$ which goes to $0$ as $\eps$ goes to $0$.}
			This allows us to couple the $i$-population sizes ($0 \leq i \leq L-1$) with sub-critical birth and death processes with inhomogeneous 
			immigration in order to approximate their time to extinction.
			
			To prove \eqref{eqpopM}, we need to control the dynamics of two types of populations: first the $i$-populations sizes, with 
			$0 \leq i \leq L-1$, which are counter-selected, and whose initial size 
			is smaller than $O(\eps^2 K)$; second the $L$-population size. 
			Let us show that with a probability converging 
			to 1 as $K\to\infty$ , $W_\eps^K < V_\eps^K$. To this aim, notice that on the time interval $[0,V_\eps^K]$, 
			the death rate of the 
			$i$-population ($0\leq i \leq L-1$) satisfies:
			\be
			d_i + \sum_{j=0}^L \frac{c_{ij}}{K}X_j \geq d_i + c_{iL}(\bar{x}_L-2\eps).
			\ee
			Moreover we know that almost surely on the event $\mathcal{E}$, we have $X_0(U_\eps^K) \leq \eps^2 K$. Hence if we
			introduce, for $0 \leq i \leq L-1$ and $k \in \N$, the notation
			\be
			T_k^{(X_i)}:= \inf \{ t \geq U_\eps^K, X_i(t) = k \}, 
			\ee
			and  apply \eqref{atteindre}, we can compare the $i$-population process to a subcritical birth-death process with the effective 
			death rate given above and obtain
			\begin{align} \label{TepskVepsK} 
				\P\left(T_{\eps K}^{(X_0)}< V_\eps^K\Big| \mathcal{E}\right) \leq &   
				\frac{((d_0 +c_{0L}(\bar{x}_L-2\eps))/b_0)^{\eps^2 K}-1}{((d_0 +c_{0L}(\bar{x}_L-2\eps))/b_0)^{\eps K}-1} 
				\leq   \left( \frac{b_0}{d_0 +c_{0L}(\bar{x}_L-2\eps)} \right)^{\eps K (1-\eps)} {\leq  C \eps},
			\end{align}
			for any constant $C$, $K$ large enough and $\eps$ small enough.

			Let us denote by $\mathcal{M}_{01}$ the number of type $1$ mutants produced by type $0$-individuals during the 
			time interval $[U_\eps^K,V_\eps^K]$. From \eqref{defN}, we have: 
			\be
			\mathcal{M}_{01}=\int_{U_\eps^K}^{V_\eps^K \wedge T_0^{(X_0)}} \mathbf{1}_{\theta\leq \mu b_0X_0(s^-)}Q_0^{(m)}(ds,d\theta). 
			\ee
			Moreover, {considering} all the possible {orderings} of $T_{\eps K}^{(X_0)}$, $V_\eps^K$, $ T_{0}^{(X_0)}$ and 
			$\ln K/\sqrt{\eps}$,  we get
			$$
			\P \left( \left\{ T_{\eps K}^{(X_0)}< V_\eps^K \wedge  T_{0}^{(X_0)} \right\} \cup 
			\left\{ V_\eps^K \wedge  T_{0}^{(X_0)} <  T_{\eps K}^{(X_0)} \wedge \frac{\ln K}{\sqrt{\eps}} \right\}\cup 
			\left\{ \frac{\ln K}{\sqrt{\eps}}< V_\eps^K \wedge  T_{0}^{(X_0)} <  T_{\eps K}^{(X_0)}  \right\}\right)=1.
			$$
			Hence, using the Markov inequality, \eqref{TepskVepsK}, as well as the fact that a subcritical branching process takes a time of 
			order $\ln K$ to get extinct (see \eqref{equi_hitting}), 
			we get that
			\begin{align}\nonumber \P \left( \mathcal{M}_{01}>K \mu \ln K | \mathcal{E} \right) 
				& \leq \frac{(\eps K) (b_0 \mu) (\ln K/\sqrt{\eps}) }{K \mu \ln K }+
				\P\left(T_{\eps K}^{(X_0)}< V_\eps^K\Big| \mathcal{E}\right)
				+\P\left(\frac{\ln K}{\sqrt{\eps}}< V_\eps^K \wedge  T_{0}^{(X_0)} <  T_{\eps K}^{(X_0)} \Big| \mathcal{E}\right)\\
				& = o_\eps(1).
			\end{align}
			
			Applying again \eqref{atteindre}, we find  that each mutant of type $1$ that is 
			produced by a type $0$ individual generates a type $1$ 
			population whose size has a probability to reach $\eps /\mu \ln K $ that is bounded by
			$$ \left( \frac{b_1}{d_1 +c_{1L}(\bar{x}_L-2\eps)} \right)^{\eps/\mu \ln K -1}. 
			$$
			We deduce that
			$$
			\P(T_{\eps K}^{(X_1)}<V_\eps^K| \mathcal{E} ) \leq \P \left( \mathcal{M}_{01}>K \mu \ln K | \mathcal{E} \right) 
			+ K\mu \ln K \left( \frac{b_1}{d_1 +c_{1L}(\bar{x}_L-2\eps)} \right)^{\eps/\mu \ln K -1}= o_\eps(1). 
			$$
			We reiterate the reasoning for the other counter-selected mutant populations ($i$-populations with $2 \leq i \leq L-1$) 
			to conclude
			\begin{equation} \label{WpluspetitqueV}
				\P\left(W_\eps^K<V_\eps^K\Big| \mathcal{E} \right)= o_\eps(1).
			\end{equation}
			By a direct application of Lemma \ref{Th3cChamp}, we get the existence of a positive constant $V$ such that:
			\begin{equation} \label{champa_grandes_dev} \liminf_{K \to \infty} \P\left( \eee^{KV} <  V_\eps^K \Big|  V_\eps^K \leq W_\eps^K , \mathcal{E}\right)=1. \end{equation}
			Using \eqref{WpluspetitqueV} and \eqref{champa_grandes_dev}, we get:
			\begin{align}\nonumber
				\liminf_{K \to \infty} \P\left( \eee^{KV} < \left( W_\eps^K \wedge  V_\eps^K \right) \Big| \mathcal{E}\right)
				& \geq \liminf_{K \to \infty} \P\left( \eee^{KV} <  V_\eps^K \leq  W_\eps^K  \Big| \mathcal{E}\right)\\ 
				& = \liminf_{K\to \infty} \P\left( \eee^{KV} <  V_\eps^K \Big| V_\eps^K \leq  W_\eps^K , \mathcal{E}\right)\P \left(V_\eps^K \leq  W_\eps^K  \Big| \mathcal{E}\right) = 1-o_\eps(1).\end{align}
			This proves \eqref{eqpopM}, {and thus statement \eqref{eq3th} of Theorem \ref{pro_phase1_mugrand}}, as we recall that $\P(\mathcal{E})\geq 1-\eps$.\\

			We may now approximate the growth rates of the $i$-population sizes 
			($0\leq i \leq L-1$) during the time interval $[U_\eps^K,V_\eps^K \wedge W_\eps^K]$. For $0 \leq i \leq L-1$, let us introduce, for $* \in \{-,+\}$
			\begin{equation} \label{defsigma-} -\sigma_i^{(*)}:=b_i(1-\mu)-d_i -c_{iL}(\bar{x}_L\bar{*}2\eps)-\mathbf{1}_{\{*=-\}}\sup_{0 \leq k \leq L-1}c_{ik}\eps, \end{equation}
			where $\bar{*}= \{-,+\} \smallsetminus * $.
			Notice that,
			for $\eps$ small enough the $(\sigma_i^{(*)})_{0 \leq i \leq L-1}$ are
			pairwise distinct by the {fourth} point of Assumption \ref{ass.1}. 
			We  consider such {an} $\eps$ throughout the remainder of the proof 
			to make sure that we do not  divide by $0$.
			Notice also that Equation \eqref{defsigma-} ensures that
			there exists $C>0$ such that, for $\eps$ small enough,
			$$
			0<|f_{iL}|-C\eps<\sigma_i^{(+)}<|f_{iL}|<\sigma_i^{(-)}<|f_{iL}|+C\eps.
			$$
			From the definition of the process $X$ in \eqref{defN} {and from \eqref{onmathcalE}}, we get that almost surely on the event $\mathcal{E}$ and for $0\leq i \leq L-1$,
			\begin{equation}\label{compaPi+-} P_i^{(-)}(t)\leq X_i(U_\eps^K+t)\leq P_i^{(+)}(t), 
				\quad \forall \ U_\eps^K\leq {U_\eps^K+} t \leq V_\eps^K \wedge W_\eps^K \end{equation}
			where, for $t\geq 0$ and $* \in \{-,+\}$,
			\begin{align}
				\label{defPi}\nonumber
				P_i^{(*)}(t)={X_i(U_\eps^K)}&+  \int_{U_\eps^K}^{U_\eps^K+t}\int_{\R_+} \mathbf{1}_{\theta \leq (1- \mu)b_{i}  P_i^{(*)}(s^-)} 
				Q_i^{(b)}(ds,d\theta)
				+ \int_{U_\eps^K}^{U_\eps^K+t}\int_{\R_+} \mathbf{1}_{\theta \leq \mu b_{i-1}  P_{i-1}^{(*)}(s^-) }Q_{i}^{(m)}(ds,d\theta)
				\\
				&- \int_{U_\eps^K}^{U_\eps^K+t}\int_{\R_+} \mathbf{1}_{\theta  \leq ((1-\mu)b_i+\sigma_i^{(*)})P_i^{(*)}(s^-)}
				Q_i^{(d)}(ds,d\theta),
			\end{align}
			where we recall that by convention $b_{-1}=0$.
			
			To find a lower bound of the extinction time of the unfit mutant population size, let us introduce
			\begin{equation} \label{defbetaM} \beta_L:= \left\{k\in\Ninterval{0}{L-1} {\text{ such that }} \frac{|f_{kL}|}{1-k/\alpha}= 
				\inf_{0\leq j \leq L-1}
				\left\{ \frac{|f_{jL}|}{1-j/\alpha} \right\} \right\} .
			\end{equation}
			We will see that the $\beta_L$-population is the one which takes the longest time to get extinct, and drives the time to 
			extinction of the whole 
			mutant-population.
			Recalling \eqref{onmathcalE}, we know that on the event $\mathcal{E}$ the  {size at time $U_\eps^K$} of the 
			$\beta_L$-population is
			$$
			C(\eps,K)K{\mu^{\beta_L}}, \quad \text{with} \quad {\mu^{\eps}} \leq C(\eps,K) \leq {\mu^{-\eps}.} 
			$$
			From \eqref{compaPi+-} and \eqref{defPi} we see that almost surely on $\mathcal{E}$ and on the time interval 
			$[U_\eps^K , V_\eps^K \wedge W_\eps^K]$,
			the $\beta_L$-population size is larger than a sub-critical birth and death process with initial state 
			$C(\eps,K)K{\mu^{\beta_L}}$, 
			individual birth rate $b_{\beta_L}(1-\mu)$, and individual death rate $b_{\beta_L}(1-\mu)+ \sigma_{\beta_L}^{(-)}$. 
			Applying {Equation \ref{equi_hitting0}}, we deduce that
			\begin{equation}\label{mintext}
				\liminf_{K\to \infty}\P \left( \inf \left\{ t \geq {0}, X_{\beta_L}({U_\eps^K+}t)=0 \right\} \geq 
				\left(1 -\frac{\beta_L}{\alpha}\right) 
				\frac{(1-\eps)}{\sigma_{\beta_L}^{(-)}}\ln K 
				\Big|\mathcal{E}\right) \geq 1-\eps.
			\end{equation}
			
			The last step of the proof consists in finding a bound {for} $\E[P_i^{(+)}(t)]$ for large $t$, to show that the total unfit mutant population 
			size takes a time of order at {most} $(1+l\eps)(1-\beta_L/\alpha)\ln K / \sigma^{+}_{\beta_L}$, for 
			some positive $l$ 
			{(to be made precised later, see \eqref{defl})}
			to get extinct. To simplify notations, 
			let us introduce, for $0 \leq i \leq L-1$ 
			and $* \in \{-,+\}$,
			\begin{equation}\label{defmathcalSi*} \mathfrak{f}_i^{(*)}:= \inf \{ \sigma_{j}^{(*)},
				0\leq j \leq i\}. \end{equation}
			We will see that the mutant population whose size decreases the slowest provides the leading term and 
			scale the time needed for all but the $L$ populations to get extinct.
			To prove that, we  now show by induction that there exists $\eps_0>0$ and a sequence of positive functions, 
			$(g_0:x\mapsto x^2,g_1,...,g_{L-1})$,
			such that, for every  $0\leq i \leq L-1$, $\eps \leq \eps_0$ and 
			$t \geq 0$, 
			\begin{equation} \label{majEPi+} \E[P_i^{(+)}(t)]\leq g_i(\eps) K {\mu^{i-\eps}}\eee^{-\mathfrak{f}_i^{(+)} t.} \end{equation}
			For $i=0$, from definitions \eqref{defsigma-}, \eqref{defPi} and property \eqref{onmathcalE}, we get
			$$
			\E[P_0^{(+)}(t)]\leq \eps^2 {K} \eee^{-\sigma_0^{(+)}t}=\eps^2 {K} \eee^{-\mathfrak{f}_0^{(+)}t} .
			$$
			Let us assume that \eqref{majEPi+} holds for every $i$ such that $0 \leq i \leq i_0 <L-1$. Then from \eqref{defsigma-}, \eqref{defPi} and the 
			induction hypothesis, for $t \geq 0$,
			\begin{align}\nonumber
				\frac{d}{dt}\E[P_{i_0+1}^{(+)}(t)]&\leq - \sigma_{i_0+1}^{(+)}\E[P_{i_0+1}^{(+)}(t)]+ \mu b_{i_0} \E[P_{i_0}^{(+)}(t)]\\\nonumber
				&\leq - \sigma_{i_0+1}^{(+)}\E[P_{i_0+1}^{(+)}(t)]+ \mu b_{i_0} g_{i_0}(\eps) K {\mu^{i_0-\eps}}\eee^{-\mathfrak{f}_{i_0}^{(+)} t}\\
				&= - \sigma_{i_0+1}^{(+)}\E[P_{i_0+1}^{(+)}(t)]+ b_{i_0} g_{i_0}(\eps) K{\mu^{i_0+1-\eps}}\eee^{-\mathfrak{f}_{i_0}^{(+)} t}.
			\end{align}
			Applying the method of variation of parameters, we get, for every $t\geq 0$,
			\begin{align}\nonumber
				\E[P_{i_0+1}^{(+)}(t)]&\leq \E[P_{i_0+1}^{(+)}(0)]\eee^{-\sigma_{i_0+1}^{(+)}t} + 
				\frac{b_{i_0} g_{i_0}(\eps) K{\mu^{i_0+1-\eps}}}{\sigma_{i_0+1}^{(+)}-\mathfrak{f}_{i_0}^{(+)}}
				\left(\eee^{-\mathcal{S}_{i_0}^{(+)} t}-\eee^{-\sigma_{i_0+1}^{(+)}t} \right)\\
				&\leq K{\mu^{i_0+1-\eps}} \left( \eee^{-\sigma_{i_0+1}^{(+)}t} + 
				\frac{b_{i_0} g_{i_0}(\eps)}{\sigma_{i_0+1}^{(+)}-\mathfrak{f}_{i_0}^{(+)}}
				\left(\eee^{-{\mathfrak{f}}_{i_0}^{(+)} t}-\eee^{-\sigma_{i_0+1}^{(+)}t} \right) \right),
			\end{align}
			where the last inequality is a consequence of \eqref{onmathcalE}. Hence, the $i_0+1$-population satisfies \eqref{majEPi+}, with 
			${\mathfrak{f}}_{i_0+1}^{(+)} = {\mathfrak{f}}_{i_0}^{(+)} \wedge \sigma_{i_0+1}^{(+)}$,
			according to the definition \eqref{defmathcalSi*}, and 
			$$
			g_{i_0+1}(\eps)=  1 + \frac{2b_{i_0}g_{i_0}(\eps)}{|\sigma_{i_0+1}^{(+)}-\mathfrak{f}_{i_0}^{(+)}|}. 
			$$

			Moreover, let us introduce $l>0$ such that, for $\eps$ small enough and for $0\leq i \leq L-1$, 
			{\begin{equation} \label{defl}
					\frac{1-i/\alpha+\eps/\alpha}{\mathfrak{f}_i^{(+)}}<(1+l\eps)\frac{1-\beta_L/\alpha}{\sigma_{\beta_L}^{(+)}} ,
			\end{equation}}
			(which is possible according to the definitions \eqref{defsigma-} and \eqref{defbetaM}) and define
			$$
			s_K:=  \frac{(1+l\eps)}{\sigma_{\beta_L}^{(+)}} {\ln \left( K \mu^{\beta_L} \right)} . 
			$$
			Then, applying \eqref{compaPi+-}, \eqref{majEPi+} and the Markov inequality, we get:
			\begin{align}\nonumber
				& \P\left( \exists i\in\Ninterval{0}{L-1}, X_i(s_K)\geq 1 \Big| \mathcal{E}\right)  \leq 
				\sum_{0\leq i \leq L-1} \P\left( X_i(s_K)\geq 1 \Big| \mathcal{E}\right) \\ \nonumber
				& \qquad     \leq  \sum_{0\leq i \leq L-1} \P\left( P_i^{(+)}(s_K)\geq 1 \Big| \mathcal{E}\right)
				\leq  \sum_{0\leq i \leq L-1} \E\left( P_i^{(+)}(s_K) \Big| \mathcal{E}\right)
				\leq  \sum_{0\leq i \leq L-1} g_i(\eps) K{\mu^{i-\eps}}\eee^{-{\mathfrak{f}}_i^{(+)} s_K}\\\nonumber
				&\qquad = \sum_{0\leq i \leq L-1} g_i(\eps)\ {\exp \left[ \mathfrak{f}_i^{(+)} \left( 
					\frac{ 1}{\mathfrak{f}_i^{(+)}}\ln \left(K\mu^{i-\eps}\right)
					-\frac{(1+l\eps)}{\sigma_{\beta_L}^{(+)}} \ln \left( K \mu^{\beta_L} \right)\right) \right]} \\
				& \qquad={ \sum_{0\leq i \leq L-1} g_i(\eps) \ \exp \left[ \mathfrak{f}_i^{(+)} \left( \frac{1-i/\alpha+\eps/\alpha}{\mathfrak{f}_i^{(+)}}
					-(1+l\eps)\frac{1-\beta_L/\alpha}{\sigma_{\beta_L}^{(+)}}\right) \ln K \left( 1 + o (1) \right) \right]},
			\end{align}
			{where we used \eqref{muenpuissancedeK} in the last line.}
			According to the definition of $l$, the last term goes to $0$ as $K$ goes to infinity.

			Combining \eqref{firstpropE}, \eqref{eqpopM}, \eqref{mintext}, and \eqref{temps_expo} {proves statement \eqref{eq2th} of Theorem \ref{pro_phase1_mugrand} and thus ends the proof of this theorem.}
		\end{proof}
		
		\subsection{Proof of Theorem \ref{pro_mupetit}}
		
		Assume first that \eqref{muenpuissancedeK} holds and that $\alpha \notin \N$. Theorem \ref{pro_mupetit} addresses the case where
		$K\mu^{L}$ is small. Only the $\lfloor\alpha \rfloor$ first mutant populations  has a large size, as
		$$
		K \mu^{\lfloor \alpha \rfloor}= {f^{\lfloor\alpha \rfloor}(K)} K^{1-{\lfloor \alpha \rfloor}/\alpha} \to \infty, \quad K \to \infty, 
		$$
		$$
		K \mu^{\lfloor \alpha \rfloor+1}= {f^{\lfloor \alpha \rfloor+1}(K)} K^{1-{(\lfloor \alpha \rfloor+1)}/\alpha}\to 0, \quad K \to \infty. 
		$$
		For $\lfloor \alpha \rfloor +1 \leq i \leq L-1$, the $i$-mutant population sizes  perform excursions
		until a successful $L$-individual is created. By successful $L$-individual we mean a 
		mutant $L$ which generates  
		a population out-competing the other populations.
		Here again the key tools are couplings with birth and death processes without 
		competition.
		
		Let us denote by 
		$T^{(i)}$ (see definition in \eqref{defTi}) the birth time of the $i$-th mutant of type ($\lfloor \alpha \rfloor +1$) 
		descended from an individual of type $\lfloor \alpha \rfloor$ and by
		$X_0^{(i)}$ the type ($\lfloor \alpha \rfloor +1$)-population generated by this individual.
		Then, we use the lexicographic order to number the $k$-mutant populations, with $\lfloor \alpha \rfloor +2 \leq k \leq L$ (see Figure \ref{fig_arbres} for 
		an illustration). More precisely, 
		\begin{itemize}
			\item For $j \geq 1$, $X^{(i)}_j$ is the ($\lfloor \alpha \rfloor +2$)-population generated by the $j$th 
			($\lfloor \alpha \rfloor +2$)-mutant produced by an
			individual of type ($\lfloor \alpha \rfloor +1$) belonging to the population $X_0^{(i)}$ 
			\item For $j,k \geq 1$, $X^{(i)}_{jk}$ is the ($\lfloor \alpha \rfloor +3$)-population generated by the $k$th 
			($\lfloor \alpha \rfloor +3$)-mutant produced by an
			individual of type ($\lfloor \alpha \rfloor +2$) belonging to the population $X^{(i)}_j$...
		\end{itemize}
		As we will see along the proof, a mutant population of type $i$ produces typically no 
		$(i+1)$-mutant, 
		one $(i+1)$-mutant with a probability of order 
		$\mu$, and more than {one $(i+1)$-mutant} with a probability of order $\mu^2$.
		The law of all trees can be approximated by the law of a sub-critical Galton-Watson process, and trees are 
		approximately independent.
		Hence we will be able to approximate the probability for the $X^{(i)}_0$ populations ($i\geq 1$) to generate a successful mutant 
		$L$ by a common probability, 
		and the time needed for a successful $L$-mutant to appear is close to an exponential random variable with mean  
		one divided by this probability.
		
		\begin{figure}[h!]
			\centering
			\includegraphics[width=.95\textwidth]{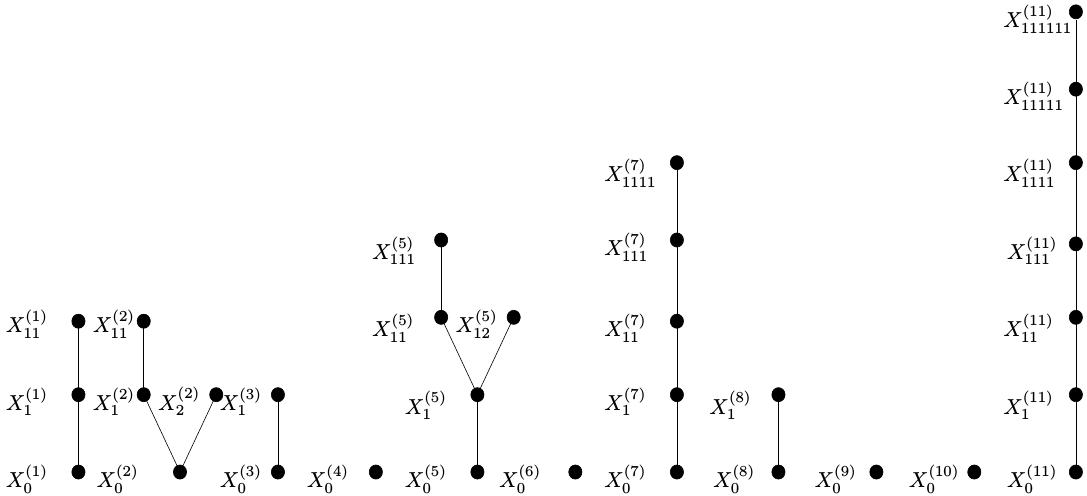}
			\caption{Numbering of the successive sub-populations: {in this example, the first ($\lfloor \alpha \rfloor +1$)-mutant 
					generates the ($\lfloor \alpha \rfloor +1$)-population $X_0^{(1)}$. One individual of the population $X_0^{(1)}$ 
					gives birth to an ($\lfloor \alpha \rfloor +2$)-mutant, which generates the 
					($\lfloor \alpha \rfloor +2$)-population $X_1^{(1)}$.
					One individual of the population $X_1^{(1)}$ 
					gives birth to an ($\lfloor \alpha \rfloor +3$)-mutant, which generates the 
					($\lfloor \alpha \rfloor +3$)-population $X_{11}^{(1)}$. The population $X_{11}^{(1)}$ gets extinct without giving birth 
					to any ($\lfloor \alpha \rfloor +4$)-individual.
					The second ($\lfloor \alpha \rfloor +1$)-mutant produced by an $\lfloor \alpha \rfloor$-individual
					generates the ($\lfloor \alpha \rfloor +1$)-population $X_0^{(2)}$. 
					Two individuals of the population $X_0^{(2)}$ 
					give birth to an ($\lfloor \alpha \rfloor +2$)-mutant. These mutants generate the 
					($\lfloor \alpha \rfloor +2$)-populations $X_1^{(2)}$ and $X_2^{(2)}$, respectively. And so on with the notations previously 
					introduced.}}
			\label{fig_arbres}
		\end{figure}
		
		{Recall the definition of the process $R_{\lfloor \alpha \rfloor}$ in \eqref{defRi}. 
			Then the stopping time $T^{(i)}$ which is the birth time of the $i$th ($\lfloor \alpha \rfloor+1$)-mutant produced by an 
			$\lfloor \alpha \rfloor$-individual 
			can be expressed as
			\begin{equation}\label{defTi} T^{(i)}:= \inf \{t \geq 0, R_{\lfloor \alpha \rfloor}(t) \geq i\}. \end{equation}}
		In particular, from \eqref{coupling1}, we get, for every $t \leq T_\eps^K \wedge S_\eps^K$,
		\begin{equation} \label{coupling2} { R_{\lfloor \alpha \rfloor}^{(-)}(t) 
				\leq R_{\lfloor \alpha \rfloor}(t) \leq R_{\lfloor \alpha \rfloor}^{(+)}(t)}, \quad a.s. ,\end{equation}
		where {processes $R_{\lfloor \alpha \rfloor}^{(\pm)}$ have been defined in \eqref{defRi}.
			Let us first, for the sake of simplicity, replace the processes $R_{\lfloor \alpha \rfloor}^{(\pm)}$ by the 
			processes $\bar{R}_{\lfloor \alpha \rfloor}^{(\pm)}$, defined in \eqref{defbarRi}, and introduce
			$$
			T^{(i,\pm)}:= \inf \{t \geq 0, \bar{R}_{\lfloor \alpha \rfloor}^{(\pm)}(t) \geq i\}.
			$$
			We will prove later on that this does not modify the result.}

		Let $u_K$ be a sequence such that 
		$$ u_K \gg \ln K \quad \text{and} \quad \mu^{\lfloor \alpha \rfloor+1}Ku_K \underset{K \to \infty}{\to} 0.  
		$$
		Using {Markov Inequality and Lemma \ref{lemme_bounds_ni}, we get
			\begin{align}\nonumber
				\P\left(T^{(1,-)}\leq u_K\right)
				& \leq \P \left( R^{(-)}\left(u_K\right)\geq 1 \right)\leq 
				\mu b_{\lfloor \alpha \rfloor} u_K\bar{x}^{(+)}_{\lfloor \alpha \rfloor}= (1+\eps)^{\lfloor \alpha \rfloor}
				\frac{b_0... b_{{\lfloor \alpha \rfloor}} x_0^{(+)}}{s_{10}^{(+)}...s_{{\lfloor \alpha \rfloor}0}^{(+)}}
				\mu^{\lfloor \alpha \rfloor+1}Ku_K \to 0, \quad (K \to \infty).
		\end{align}}

		Following the ideas developed in Section \ref{alpha+1geqM}, we may couple each $k$-mutant 
		population $(\lfloor \alpha \rfloor +1\leq k \leq L-1)$ with two 
		birth and death processes, independent conditionally on their birth time. We will not detail the couplings 
		as the ideas have already been developed and 
		the notations are tedious, but we nevertheless state rigorously the resulting properties. 
		Let us denote by $(T_\mathfrak{j}^{(i)}, \mathfrak{j} \in {\cup_{n\in\N}\N^n}, i \in \N)$ the time of appearance of the populations
		$(X_\mathfrak{j}^{(i)}, \mathfrak{j} \in {\cup_{n\in\N}\N^n}, i \in \N)$. For all  $\mathfrak{j} \in {\cup_{n\in\N}\N^n}, i \in \N$, 
		$ 
		T_\mathfrak{j}^{(i)}:= \inf \{ t \geq 0, X_\mathfrak{j}^{(i)}(t) \geq 1 \}$.
		Then we  introduce birth and death processes 
		$(X_\mathfrak{j}^{(i,*)}, \mathfrak{j} \in {\cup_{n\in\N}\N^n}, i \in \N, * \in \{-,+\})$
		with  birth and death rates 
		$$
		\left(\left(b_{\mathfrak{t}(\mathfrak{j})},
		(1-\mu)b_{\mathfrak{t}(\mathfrak{j})}+\sigma_{\mathfrak{t}(\mathfrak{j})}^{(*)}\right), \mathfrak{j} \in 
		\cup_{n\in\N}\N^n, i \in \N, * \in \{-,+\}\right),
		$$
		where {the $\sigma^{(*)}$'s} have been defined in \eqref{defsigma-},
		$$
		\mathfrak{t}(\mathfrak{j}):=\lfloor \alpha \rfloor + |\mathfrak{j}|+1,
		$$
		and $|\mathfrak{j}|$ is the number of terms in $\mathfrak{j}$ (for instance $|11221|=5$).
		
		These processes can be constructed in such a way that, for all $ \mathfrak{j} \in {\cup_{n\in\N}\N^n}, i \in \N$,
		\begin{equation} \label{couplage_arbres} 
			X_\mathfrak{j}^{(i,-)}(t)\leq X_\mathfrak{j}^{(i)}(t)+\mathfrak{N}_\mathfrak{j}^{(i)}(t) \leq X_\mathfrak{j}^{(i,+)}(t),
			\quad t \leq T_\eps^K \wedge S_\eps^K,\end{equation}
		where $\mathfrak{N}_\mathfrak{j}^{(i)}(t)$ is the number of mutants of type ($\lfloor \alpha \rfloor + |\mathfrak{j}|+2$) produced by the 
		$X_\mathfrak{j}^{(i)}$ population (which is of type ($\lfloor \alpha \rfloor + |\mathfrak{j}|+1$)) until time $t$.
		Recall that among the offsprings produced by the population $X_\mathfrak{j}^{(i)}$, a fraction $(1-\mu)$ is constituted by newborn individuals of type
		$\lfloor \alpha \rfloor + |\mathfrak{j}|+1$, and a fraction $\mu$ by new born individuals of type $\lfloor \alpha \rfloor + |\mathfrak{j}|+2$, and that 
		at each birth event the probability to have a mutation is independent from the past.
		
		Moreover, conditionally on $(T_\mathfrak{j}^{(i)}, \mathfrak{j} \in {\cup_{n\in\N}\N^n}, i \in \N)$, the pairs
		of processes 
		$((X_\mathfrak{j}^{(i,-)},X_\mathfrak{j}^{(i,+)}), \mathfrak{j} \in {\cup_{n\in\N}\N^n}, i \in \N)$
		are independent.
		This last point allows us to approximate the law of the random trees 
		$$\mathcal{T}^{(i)}:= X_0^{(i)} \cup_{n \in \N} X^{(i)}_{\N^n}, \quad i \in \N $$
		(an example is depicted in Figure \ref{fig_arbres}) by the same law, and independently.
		Indeed from Equation \eqref{couplage_arbres}, we get that $(\mathcal{T}^{(i)})_{i \in \N}$, can be coupled with two inhomogeneous birth and death processes, 
		whose law is well known and easy to study. This will be the object of the end of the proof.\\
		
		We will now consider each tree $\mathcal{T}^{(i)}$ ($i \in \N$) with root $X_0^{(i)}$ independently, and approximate its probability to end with a $L$-mutant
		sub-population. For simplicity we write $|0|=0$.
		
		Consider a vertex $X_\mathfrak{j}^{(i)}$, $\mathfrak{j} \in \{0\}\cup \N^\N$ of the tree $\mathcal{T}^{(i)}$.
		Due to the coupling \eqref{couplage_arbres} we are able to approximate the probability for this vertex to have no child, one child or 
		more than {one child.}
		{Before} the time $T_\eps^K \wedge S_\eps^K$, the law of the number of offsprings produced by the 
		$X_\mathfrak{j}^{(i)}$ population is smaller (resp. larger) than the number of offspring produced by a population initiated by one individual, with 
		individual birth rate $b_{\mathfrak{t}(\mathfrak{j})}$ and individual death rate $(1-\mu)b_{\mathfrak{t}(\mathfrak{j})}+\sigma^{(+)}_{\mathfrak{t}(\mathfrak{j})}$
		(resp. $(1-\mu)b_{\mathfrak{t}(\mathfrak{j})}+\sigma^{(-)}_{\mathfrak{t}(\mathfrak{j})}$). Moreover, each offspring is a mutant of type 
		($\mathfrak{t}(\mathfrak{j})+1$) with probability $\mu$, and is a clone with probability $(1-\mu)$. Hence
		$$
		\P\left( X_\mathfrak{j}^{(i)} \text{-pop produces $1$ mutant} \right) \leq \sum_{n=0}^\infty p^{(b_{\mathfrak{t}(\mathfrak{j})},(1-\mu)b_{\mathfrak{t}(\mathfrak{j})}+\sigma^{(+)}_{\mathfrak{t}(\mathfrak{j})})}(n)
		n \mu (1-\mu)^{n-1} \leq \mu \eee^{(b_{\mathfrak{t}(\mathfrak{j})},b_{\mathfrak{t}(\mathfrak{j})}+\sigma^{(+)}_{\mathfrak{t}(\mathfrak{j})})},
		$$
		where $p^{(.,.)}$ and $\eee^{(.,.)}$ are defined in Lemma \ref{lem_taille_excu}. 
		Similarly, for $K$ large enough,
		\begin{align}
			\P\left( X_\mathfrak{j}^{(i)} \text{-pop produces $1$ mutant} \right) &\geq \sum_{n=0}^\infty p^{(b_{\mathfrak{t}(\mathfrak{j})},(1-\mu)b_{\mathfrak{t}(\mathfrak{j})}+\sigma^{(-)}_{\mathfrak{t}(\mathfrak{j})})}(n)
			n \mu (1-\mu)^{n-1} \\\nonumber
			&\geq \mu \sum_{n=0}^{\mu^{-1/2}+1} p^{(b_{\mathfrak{t}(\mathfrak{j})},(1-\mu)b_{\mathfrak{t}
					(\mathfrak{j})}+\sigma^{(-)}_{\mathfrak{t}(\mathfrak{j})})}(n)
			n (1-\mu)^{\mu^{-1/2}}\\\nonumber
			&\geq \mu \sum_{n=0}^{\mu^{-1/2}+1} p^{(b_{\mathfrak{t}(\mathfrak{j})},(1-\mu)b_{\mathfrak{t}
					(\mathfrak{j})}+\sigma^{(-)}_{\mathfrak{t}(\mathfrak{j})})}(n)
			n (1-2\sqrt{\mu})\\\nonumber
			&= \mu(1-2\sqrt{\mu})\left( \eee^{(b_{\mathfrak{t}(\mathfrak{j})},(1-\mu)b_{\mathfrak{t}(\mathfrak{j})}+
				\sigma^{(-)}_{\mathfrak{t}(\mathfrak{j})})}
			-  \E\left[\mathbf{1}_{\{|X_\mathfrak{j}^{(i)}|\geq \mu^{-1/2}+1\}}|X_\mathfrak{j}^{(i)}|\right]\right),
		\end{align}
		where $|X_\mathfrak{j}^{(i)}|$ denotes the number of offsprings (mutants or clones) produced by the population $X_\mathfrak{j}^{(i)}$.
		But Cauchy-Schwarz and Markov inequalities yield
		$$
		\E^2\left[\mathbf{1}_{\{|X_\mathfrak{j}^{(i)}|\geq \mu^{-1/2}+1\}}|X_\mathfrak{j}^{(i)}|\right]\leq 
		\P\left(|X_\mathfrak{j}^{(i)}|\geq \mu^{-1/2}+1\right)\E\left[|X_\mathfrak{j}^{(i)}|^2\right]
		\leq \sqrt{\mu}\E\left[|X_\mathfrak{j}^{(i)}|\right]\E\left[|X_\mathfrak{j}^{(i)}|^2\right]= O(\sqrt{\mu}),
		$$
		as the two expectations are finite according to \eqref{loi_excursion}.
		
		Adding \eqref{diff_e}, we may conclude that, as $K$ goes to infinity,
		$$
		\P\left( X_\mathfrak{j}^{(i)} \text{-pop produces $1$ mutant} \right)= 
		\eee^{(b_{\mathfrak{t}(\mathfrak{j})},b_{\mathfrak{t}(\mathfrak{j})}+|f_{\mathfrak{t}(\mathfrak{j})0}|)}\mu(1 + O(\eps)). 
		$$
		Using again coupling \eqref{couplage_arbres} and \eqref{loi_excursion}, we get  that
		$$
		\P\left( X_\mathfrak{j}^{(i)} \text{-pop produces at least $2$ mutants} \right) 
		\leq \sum_{n=0}^\infty p^{(b_{\mathfrak{t}(\mathfrak{j})},b_{\mathfrak{t}(\mathfrak{j})}+\sigma^{(+)}_{\mathfrak{t}(\mathfrak{j})})}(n)
		\frac{n(n-1)}{2} \mu^2=O(\mu^2).
		$$
		
		From the last computations, we can infer that, for $i \geq 1$, the probability for the tree $\mathcal{T}^{(i)}$
		to produce an $L$-mutant is, for large $K$,
		\begin{equation} \label{star}
			\mu^{L-1-\lfloor \alpha \rfloor} \left(\prod_{k = \lfloor \alpha \rfloor +1}^{L-1}\eee^{(b_k,b_k+{|f_{k0}|})}\right)(1+O(\eps)).  
		\end{equation}
		Indeed, the probability for each vertex to produce one child is of order $\mu$, and the probability to produce at least two children is of order $\mu^2$. 
		Since there is only  a finite number of possible mutations, independent of $\mu$, this implies that the probability for the tree $\mathcal{T}^{(i)}$ 
		to have at least one vertex with two children and end with an $L$ individual is of order
		$ \mu^{L-\lfloor \alpha \rfloor} $,
		which is negligible compared to 
		$ \mu^{L-1-\lfloor \alpha \rfloor}$.
		Moreover, we know that each $L$-mutant has a probability close to $f_{L0}/b_L$ to generate a
		population whose size hits the 
		value $\eps K$, and once this 
		size is reached, the time needed for the $L$-population to outcompete the other populations and hit its equilibrium size is of 
		order $\ln K$ {(see for instance \cite{champagnat2006microscopic})}, which is negligible 
		with respect to the time needed for the successful $L$-individual to be born.
		{If the times of appearance of the trees $\mathcal{T}^{(i)}$ had the law of a Poisson 
			process with inhomogeneous parameter close to $\mu b_{\lfloor \alpha \rfloor}x_{\lfloor \alpha \rfloor}K$
			(that is to say if we could consider $\bar{R}^{(\pm)}_{\lfloor \alpha \rfloor}$ instead of 
			$R^{(\pm)}_{\lfloor \alpha \rfloor}$), this would end the proof of the first point of Theorem \ref{pro_mupetit}.
			We now need to justify that the result stays true when considering $R^{(\pm)}_{\lfloor \alpha \rfloor}$.}
		{To achieve this goal, it is enough to prove the existence of two sequences $N_1(K)$ and $N_2(K)$ satisfying
			\begin{equation} \label{FCgrands} N_1(K) \gg \left(K\mu^L\right)^{-1} \quad \text{and} \quad N_2(K) \ll \left(\mu^{L-1-\lfloor \alpha \rfloor}\right)^{-1} \end{equation}
			such that 
			\begin{equation} \label{pareil} \lim_{K \to \infty} 
				\P \left( \sup_{s \leq N_1(K)} \left| R^{(\pm)}_{\lfloor \alpha \rfloor}(s)-\bar{R}^{(\pm)}_{\lfloor \alpha \rfloor}(s) \right| > N_2(K) \right)=0. 
			\end{equation}
			Indeed, this implies that during the time interval under consideration (of order $(K\mu^L)^{-1}$),
			the difference between the number of 'trees' generated by the processes $ R^{(\pm)}_{\lfloor \alpha \rfloor}$ and  
			$\bar{R}^{(\pm)}_{\lfloor \alpha \rfloor}$ is much smaller than $(\mu^{L-1-\lfloor \alpha \rfloor})^{-1}$, and as 
			each tree has a probability of order $(\mu^{L-1-\lfloor \alpha \rfloor})$ to generate a successful mutant, the same tree 
			is at the origin of the successful mutant under the two counting processes under consideration with a probability close to $1$.}
		
		{To prove \eqref{pareil}, we  apply Doob's martingale inequality to $M_{\lfloor \alpha \rfloor}$. This yields:
			\begin{align*}
				\P \left( \sup_{s \leq N_1(K)} \left| R^{(\pm)}_{\lfloor \alpha \rfloor}(s)-\bar{R}^{(\pm)}_{\lfloor \alpha \rfloor}(s) \right| 
				> N_2(K) \right)
				& \leq \frac{\E \left[ \left(M_{\lfloor \alpha \rfloor}^{(\pm)}(N_1(K))\right)^2 \right]}{N_2^2(K)}
				\leq \frac{2 \mu b_{\lfloor \alpha \rfloor} \bar{x}^{(\pm)}_{\lfloor \alpha \rfloor} N_1(K)}{N_2^2(K)}\\
				&\leq C K\mu^{\lfloor \alpha \rfloor+1}  \frac{N_1(K)}{N_2^2(K)}
				%= C \frac{K^{1-(\lfloor\alpha\rfloor+1)/\alpha} N_1(K)}{N_2^2(K)}
				= C\frac{K\mu^L}{\left(\mu^{L-1-\lfloor \alpha \rfloor}\right)^{2}} \frac{N_1(K) }{N_2^2(K)}\mu^{L-1-\lfloor \alpha \rfloor},
			\end{align*}
			where $C$ is a finite constant. As $\mu^{L-1-\lfloor \alpha \rfloor}$ goes to $0$ when $K$ tends to $\infty$, the sequences $N_1(K)$ and $N_2(K)$ can be chosen 
			in such a way that the last term in the previous series of inequalities goes to $0$ when $K$ tends to $\infty$, which ends the proof of \eqref{pareil}.}\\
		
		To end the proof of Theorem \ref{pro_mupetit} let us consider the case when $\mu \ll 1/K$.
		From Lemma \ref{Th3cChamp} we know that, for $\eps$ small enough, there exists a positive $V$ such that with high probability, the size of a monomorphic $0$-population 
		stays at a distance smaller than $\eps K$ from its equilibrium size $\bar{n}_0 K$ during a time larger than $\eee^{KV}$. 
		As a consequence, if $K\mu \gg \eee^{-VK}$, the $0$-population produces a large number of $1$-mutants during the time interval $[0,\eee^{VK}]$, with a rate very close to 
		$b_0 \bar{n}_0 K \mu $.
		Hence the proof is very similar to the previous proof, where the 
		$\lfloor \alpha \rfloor$-population is replaced by the $0$-population.
		
		\section{Proofs of Section \ref{res_ext}}
		
		\subsection{Proof of Theorem {\ref{theo_range} point 2}}
		
		Recall from \eqref{defN} that the process $X_0$ admits the following Poisson representation:
		\begin{equation} \label{defX0} 
			X_0(t)={\lfloor\bar{x}_0 K\rfloor}+ \int_0^t\int_{\R_+}  \mathbf{1}_{\theta \leq (1- \mu)b_{0}  X_0(s^-)} 
			Q_0^{(b)}(ds,d\theta)
			- \int_0^t\int_{\R_+}  \mathbf{1}_{\theta  \leq D^K_0(X(s^-))X_0(s^-)}
			Q_0^{(d)}(ds,d\theta),
		\end{equation}
		where $ D^K_0(X)$ {is} defined in \eqref{deathrate}.
		Thus, if we introduce the process $Y_0$ via
		$$
		Y_0(t)={\lfloor \bar{x}_0 K \rfloor} + \int_0^t\int_{\R_+}  \mathbf{1}_{\theta \leq b_{0}  Y_0(s^-)} 
		Q_0^{(b)}(ds,d\theta)
		- \int_0^t\int_{\R_+}  \mathbf{1}_{\theta  \leq (d_0+c_{00}Y_0(s^-)/K)Y_0(s^-)}
		Q_0^{(d)}(ds,d\theta),
		$$
		we get that, almost surely, 
		$ X_0(t) \leq Y_0(t),$ for all $ t \geq 0$.
		Now consider a time $v_K$ such that
		$$
		\frac{1}{\rho_0(K)} \ll v_K \ll \frac{1}{K \mu^L}, \quad K \to \infty.
		$$
		where $\rho_0(K)$ was defined in \eqref{defCpara}.
		If we apply inequality (3.7) of \cite{chazottes2016sharp} to the process $Y_0$, we get:
		$$
		d_{\text{TV}}\left(\P(Y_0(v_K) \in .), \delta_0(.)\right) \underset{K \to \infty}{\to} 0,  
		$$
		where $d_{\text{TV}}$ is the total variation distance. This implies
		\begin{equation} \label{mortX0} \P(X_0(v_K)>0)\underset{K \to \infty}{\to} 0. \end{equation}
		Hence to prove Theorem {\ref{theo_range} point 2} it is enough to show that 
		$ \P(B_L < v_K)\underset{K \to \infty}{\to} 0$.
		Notice that from \eqref{defX0} we have, for every positive $t$,
		$$
		\frac{d}{dt}\E \left[X_0(t)\right] \leq \E \left[ \left( b_0-d_0-\frac{c_{00}}{K}X_0(t) \right)X_0(t) \right] 
		\leq (b_0-d_0) \E[X_0(t)] - \frac{c_{00}}{K} \E^2[X_0(t)].
		$$
		Thus, for all $t \geq 0$, we have
		$\E[X_0(t)] \leq \bar{x}_0 K$.
		Next we 
		bound the expectation of the total number $\Xi_1$ of type $1$ individuals generated by type $0$ individuals by mutations before the time $v_K$:
		\begin{equation} \label{majEXii} 
			\E[\Xi_1] \leq \int_0^{v_K} b_0 \mu \E[{X_0(s)}]ds \leq b_0 \bar{x}_0 K \mu v_K. 
		\end{equation}
		We want to bound the probability that at least one type $1$ individual born from a type $0$ individual before time $v_K$ has a line of descent containing 
		a type $L$ individual. Denote by $\xi_i$ the event that  the $i$th type $1$ individual born from a type $0$ individual before time $v_K$ has a descendant of type $L$ 
		at any time in the future.
		We see that 
		$$
		\P(B_L < v_K) = \P\left( \underset{{i \leq {\Xi_1}}}{\cup} \xi_i \right) =  \E \left[\P \left( \underset{{i \leq {\Xi_1}}}{\cup} \xi_i \Big|
		{\Xi_1} \right) \right]
		\leq  \E \left[ \underset{{i \leq {\Xi_1}}}{\sum}  \P \left( \xi_i \Big| {\Xi_1} \right) \right].
		$$
		But recall that by Assumption \ref{A4}, for $1 \leq i \leq L-1$, $b_i<d_i$. Hence using \eqref{star}, we see that the probability of the events 
		$(\xi_i)_{1 \leq i \leq {\Xi_1}}$ can be bounded independently of ${\Xi_1}$ by 
		$$
		2\left(\underset{1 \leq i \leq L-1}{\prod}\eee^{(b_i,d_i)}\right)\mu^{L-1}. 
		$$
		This yields 
		$$
		\P(B_L < v_K) \leq b_0 \bar{x}_0  v_K \left(\underset{1 \leq i \leq L-1}{\prod}\eee^{(b_i,d_i)}\right)\mu^{L}\underset{K \to \infty}{\to} 0. 
		$$
		Adding \eqref{mortX0} ends the proof.

		\subsection{Proof of Theorem \ref{theo_range} {point 1}}
		
		{We} introduce $v_K$ such that 
		$ \frac{1}{\rho_0(K)} \ll v_K \ll \frac{1}{K \mu}$.
		Then \eqref{majEXii} and Markov Inequality ensure that with a probability close to $1$, no type $1$ mutant is produced before the population extinction.
		As a consequence, no type $L$ mutant is produced. This ends the proof.
		
		\renewcommand\thesection{\Alph{section}}
		\setcounter{section}{0}
		\renewcommand{\theequation}{\Alph{section}.\arabic{equation}}

		\section{Technical results}
		
		The next Lemma quantifies the time spent by a birth and death process with logistic competition in a vicinity of its 
		equilibrium size. It is stated in \cite{champagnat2006microscopic} Theorem 3(c).
		
		\begin{lemma}\label{Th3cChamp}
			Let $b,d,c$ be in $\R_+^*$ such that $b-d>0$.
			Denote by $(W_t)_{t \geq 0}$ a density dependent birth and death process with birth rate $bn$ and death rate
			$(d+c n/K)n$, where $n \in \N_0$ is the current state of the process and $K \in \N$ is the carrying capacity.
			Fix $0 < \eta_1 < (b - d)/c$ and $\eta_2 > 0$, and introduce the stopping time
			$$
			\mathcal{S}_K = \inf \left\{t \geq  0: W_t \notin
			\left[\Big(\frac{b - d}{c}- \eta_1\Big)K,  \Big( \frac{b - d}{c}+\eta_2\Big)K\right]\right\}.
			$$
			Then, there exists $V > 0$ such that, for any compact subset $C$ of $](b - d)/c -
			\eta_1 , (b - d)/c + \eta_2 [$,
			\begin{equation}\label{temps_expo}\lim_{K \to \infty}
				\underset{k/K \in C}{\sup} \P_k(\mathcal{S}_K< \eee^{KV} ) = 0.\end{equation}
		\end{lemma} 
		
		Let us now recall some results on hitting times of a birth and death process. The first, third, and last statements can be found in \cite{BP17}.
		The second statement is a consequence of the first statement.
		\begin{lemma} \label{resultatsBP}
			Let $Z=(Z_t)_{t \geq 0}$ be a birth and death process with individual birth and death rates $b$ and $d $. For $i \in \Z_+$, 
			$T_i=\inf\{ t\geq 0, Z_t=i \}$ and $\P_i$ (resp. $\E_i$) is the law (resp. expectation) of $Z$ when $Z_0=i$. Then 
			\begin{enumerate} 
				\item[$\bullet$] If $d\neq b \in \R_+^*$, for every $i\in \Z_+$ and $t \geq 0$,
				\begin{equation} \label{ext_times} \P_{i}(T_0\leq t )= \Big( \frac{d(1-\eee^{(d-b)t})}{b-d\eee^{(d-b)t}} \Big)^{i}.\end{equation}
				\item[$\bullet$] If $0<b<d$ and $Z_0=N$, the following convergence holds:
				\begin{equation} \label{equi_hitting0}  T_0/\log N \underset{N \to \infty}{\to} (d-b)^{-1}, \quad  \text{in probability}.  \end{equation}
				\item[$\bullet$] If $0<d<b$, on the non-extinction event of $Z$, which has a probability $1-(d/b)^{Z_0}$, the following convergence holds:
				\begin{equation} \label{equi_hitting}  T_N/\log N \underset{N \to \infty}{\to} (b-d)^{-1}, \quad  a.s.  \end{equation}
				\item[$\bullet$] If $0<b<d$, {and $(i,j,k) \in \N^3$ such that $j \in (i,k)$,
					\begin{equation}\label{atteindre}
						\P_j(T_k<T_i)= \frac{(d/b)^{j-i}-1}{(d/b)^{k-i}-1}.
				\end{equation}}
				
			\end{enumerate}
		\end{lemma}
		
		The last result of this Appendix concerns the size distribution of the total number of individuals in a sub-critical birth and death process. 
		We refer the reader to \cite{van2016random} {(Theorem 3.13 applied to the case when $X$ is a geometric 
			random variable with parameter $d/(b+d)$)} or \cite{BP17} for the proof of the two first points. The last one is just a consequence of the Mean Value Theorem.
		
		\begin{lemma} \label{lem_taille_excu}
			Let us consider a birth and death process with individual birth rate $b>0$ and individual death rate $d>0$ satisfying $b<d$.
			Let $Z$ denote the total number of births during an excursion of this process initiated with one individual.
			Then, for $k \geq 0$,
			\begin{equation}\label{loi_excursion}
				p^{(b,d)}(k):=\P(Z=k)=\frac{(2k)!}{k! (k+1)!}\left( \frac{b}{d + b} \right)^{k} \left( \frac{d}{d + b} \right)^{k+1}.
			\end{equation}
			In particular, 
			\begin{equation}\label{esp_excursion}
				\eee^{(b,d)}:=\E[Z]= \sum_{{k=1}}^\infty \frac{(2k)!}{{(k-1)!}(k+1)!}\left( \frac{b}{d + b} \right)^{k} \left( \frac{d}{d + b} \right)^{k+1}.
			\end{equation}
			Moreover, there exist two positive constants $c$ and $\eps_0$ such that, for every $\eps \leq \eps_0$, 
			if $0<d_i<b_i$ and $|b_i-d_i|\leq \eps$, $i \in \{1,2\}$, then
			\begin{equation}\label{diff_e}
				\left| \eee^{(b_1,d_1)}-\eee^{(b_2,d_2)} \right|\leq c\eps.
			\end{equation}
			
		\end{lemma}

		\begin{figure}[h!]
			\centering
			\includegraphics[width=.8\textwidth]{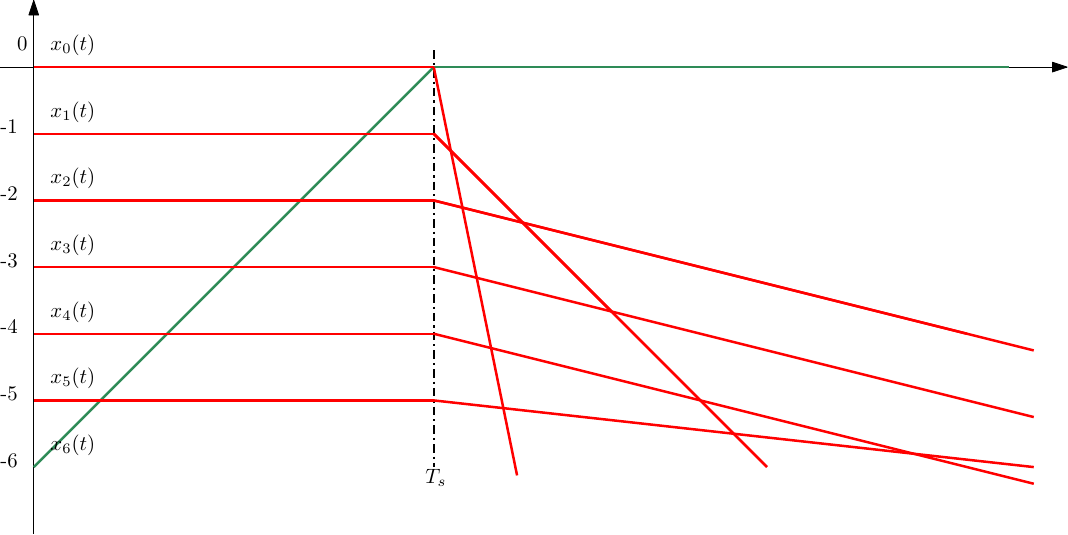}
			\caption
			{Graph of $x(t)$ in the 1-sided case $m_{ij}=m^{(1)}_{ij}$ for $L=6$ and $f_{60}=1, (f_{06},f_{16},f_{26},f_{36},f_{46},f_{56})=(-5, -1,-0.25,-1.5,-2,-0.05)$, which is the fitness landscape depicted in Figure \ref{fitness}. }
			\label{fig-pow-1sided}
		\end{figure}

		\begin{figure}[h!]
			\centering
			\includegraphics[width=.8\textwidth]{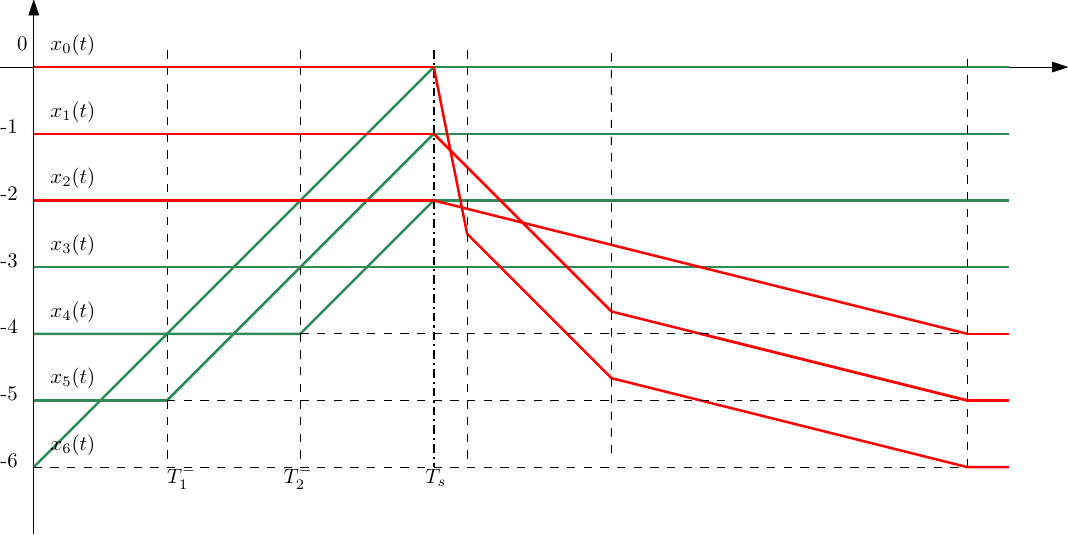}
			\caption
			{Graph of $x(t)$ in the 2-sided case $m_{ij}=m^{(2)}_{ij}$ for $L=6$ and $f_{60}=1, (f_{06},f_{16},f_{26})=(-5, -1,-0.25)$, which is (compatible with) the fitness landscape depicted in Figure \ref{fitness}. }
			\label{fig-pow-general}
		\end{figure}

		% \bibliographystyle{abbrv}
		% \bibliography{biblio_hitch}
		% 
		% 
		% \end{document}
		% 
		
		%\newpage
		
		\bibliographystyle{abbrv}
		\bibliography{Library-Biomaths}

\end{document}